\documentclass[11pt,a4paper]{article}

\usepackage[asymmetric]{geometry}
\usepackage{amssymb,amsmath,amsfonts,amsthm,verbatim,cite,subcaption,tikz,enumerate}
\usepackage[algo2e,ruled,linesnumbered]{algorithm2e}

\usepackage{thmtools}
\usepackage{thm-restate}
\usepackage{hyperref}
\hypersetup{ colorlinks = true, citecolor = {blue}, linkcolor = {magenta}}

\usepackage[capitalise,nameinlink,noabbrev]{cleveref}
\crefname{equation}{}{}
\usepackage{cases}
\usepackage{cancel}  

\newcommand{\ie}{i.e., }

\newcommand{\argmin}{\mathop{\arg\min}}

\newcommand{\Ttran}{\mathsf{T}}

\newcommand{\de}{\,\mathrm{d}}

\newcommand{\Rq}{\mathrm{Rq}}

\newcommand{\Ly}{\mathcal{L}}

\newcommand{\proj}{\Pi}
\newcommand{\defi}{\equiv}
\newcommand{\normml}{\left\vert\kern-0.25ex\left\vert\kern-0.25ex\left\vert}
\newcommand{\normmr}{\right\vert\kern-0.25ex\right\vert\kern-0.25ex\right\vert}
\newcommand{\order}{\mathcal{O}}
\newcommand{\zero}{\mathbf{0}}
\newcommand{\R}{\mathbb{R}}
\newcommand{\Sp}{\mathcal{S}_{+}^{n-1}}
\newcommand{\fig}{pdf}

\newcommand{\figsizeF}{0.24\textwidth}

\providecommand{\spa}[1]{\mathrm{span}\{#1\}}

\providecommand{\abs}[1]{\lvert#1\rvert}
\providecommand{\norm}[1]{\lVert#1\rVert}
\providecommand{\dual}[1]{\langle#1\rangle}

\providecommand{\bignorm}[1]{\bigl\lVert#1\bigr\rVert}
\providecommand{\Bignorm}[1]{\Bigl\lVert#1\Bigr\rVert}

\newtheorem{lemma}{Lemma}[section]
\newtheorem{theorem}{Theorem}[section]
\newtheorem{corollary}{Corollary}[section]
\newtheorem{remark}{Remark}[section]

\newtheorem{proposition}{Proposition}[section]

\numberwithin{equation}{section}

\usepackage{algorithm, algorithmic}

\begin{document}

\title{EPIC: a provable accelerated Eigensolver based on Preconditioning 
and Implicit Convexity}
\author{
Nian~Shao\thanks{Institute of Mathematics, EPF Lausanne, 1015 Lausanne, Switzerland (\url{nian.shao@epfl.ch}). 
}
\and
Wenbin~Chen\thanks{School of Mathematical Sciences and Shanghai Key Laboratory for Contemporary Applied Mathematics, Fudan University,  Shanghai, 200433, P. R. China (\url{wbchen@fudan.edu.cn}).}
\and
Zhaojun~Bai\thanks{Department of Computer Science and
Department of Mathematics, University of California, Davis, CA 95616, USA (\url{zbai@ucdavis.edu}).}
}

\date{}

\maketitle

\begin{abstract}
	This paper is concerned with the extraction of the smallest eigenvalue and the corresponding eigenvector of a symmetric positive definite matrix pencil. 
    We reveal implicit convexity of the eigenvalue problem in Euclidean space.
	A provable accelerated eigensolver based on preconditioning and implicit convexity (EPIC) is proposed. 
	Theoretical analysis shows the acceleration of EPIC with the rate of convergence resembling the expected rate of convergence of the well-known locally optimal preconditioned conjugate gradient (LOPCG). A complete proof of the expected rate of convergence of LOPCG is elusive so far.  Numerical results confirm our theoretical findings of EPIC.
\end{abstract}

\medskip

%


\section{Introduction}
Eigenvalue problems are cornerstones in scientific and engineering computations. In this paper, we consider the following generalized eigenvalue problem:
\begin{equation} \label{evp}
	Au=Mu\lambda,
\end{equation}
where $A$ and $M$ are given $n \times n$ symmetric positive definite matrices, and $(\lambda, u)$ is a desired eigenpair.
Numerous algorithms for computing eigenvalues and their associated eigenvectors have been developed \cite{Bai2000,Parlett1998,Saad2011,Wilkinson1965,Golub2013}. 
Preconditioning techniques are often necessary for large scale problems and have been well-studied for solving linear systems of equations \cite{Wathen2015,Benzi2002}.
For eigenvalue problems, preconditioning has also been investigated extensively.
There are classical preconditioned gradient-type eigensolvers, such as
the preconditioned steepest descent method \cite{Samokish1958,Neymeyr2012,Zhou2019,Zhou2023}
and the preconditioned gradient-type method \cite{Knyazev1994,Knyazev1998}.
The convergence analysis of these gradient-type eigensolvers are studied in \cite{Dyakonov1996,Knyazev2003,Ovtchinnikov2006,Argentati2017} and the references therein.
One of the most popular preconditioned iterative method for the eigenvalue problem \Cref{evp} is the Locally Optimal Block Preconditioned Conjugate Gradient (LOBPCG) method \cite{Knyazev2001}.
Due to the use of a momentum term, the convergence of LOBPCG is significantly accelerated under careful implementations \cite{Duersch2018,Knyazev2007}.
Despite its great success in practices, the complete proof  of the expected rate of convergence and acceleration of LOBPCG in \cite[(5.5)]{Knyazev2001} is still elusive. 

There are preconditioned eigensolvers with momentum from the perspectives of differential equations, see \cite{Chen2022} and references therein. Numerical results show that adding a momentum term can significantly improve the the rate of convergence, but theoretically, the acceleration is hard to prove.

Momentum methods are widely used in convex optimization, which can date back to early 1960s \cite{Polyak1964}.
A popular  momentum method is the Nesterov Accelerated Gradient (NAG)
flow \cite{Nesterov1983}. 
From the theoretical analysis of the convergence of NAG flows, the classical technique was the estimating sequences \cite{Nesterov2018}. 
Recently, a second-order ordinary differential equation (ODE) was derived in \cite{Su2014} to study the dynamic of NAG flows.
The connection between NAG flows and the ODEs has been studied extensively 
in the last few years \cite{Shi2021,Muehlebach2021,Luo2021,Odonoghue2015}.
For example, by combining the NAG flow with the preconditioning technique,
a preconditioned accelerated gradient descent methods for solving nonlinear PDEs was proposed in \cite{Park2021}.

 
The crux of the great success of the NAG flow approach is the convexity of objective function. 
Unfortunately, for the eigenvalue problem \eqref{evp}, the associated Rayleigh quotient 
\begin{equation*}
\Rq(x)=\frac{x^{\Ttran}Ax}{x^{\Ttran}Mx}
\end{equation*}
is not (strongly) convex in  Euclidean space, due to the homogeneity $\Rq(tx)=\Rq(x)$ for all $t\neq0$.
One way to explore the convexity in eigenvalue computation is to consider the Rayleigh quotient on smooth manifolds \cite{Edelman1998,Alimisis2022}. Recently, a Riemannian Acceleration with Preconditioning (RAP) is proposed in \cite{Shao2023b}. It is an accelerated preconditioned eigensolver with rigorous proofs of the convergence and acceleration. 
Although the convexity structure on Riemannian manifolds is well--studied, the analysis of preconditioning is involved.
Besides the spectral condition number $\kappa(T^{-1}A)$ in \cite{Knyazev2003}, where $T$ is the symmetric positive preconditioner for $A$, some extra technical conditions for preconditioners, such as the leading angle, are required for the acceleration due to the operations on manifolds.
Even though extra conditions can be verified for some popular preconditioners, such as the domain decomposition, it would be better if the acceleration can be obtained with only some requirements about the spectral condition number.
One possible strategy is exploring the implicit convexity structure in Euclidean space as we will pursue in this work.

\paragraph{Contributions.}
In this paper, we reveal implicit convexity of the eigenvalue problem \eqref{evp} with respect to the smallest eigenvalue and the corresponding eigenvector. 
Compared with the treatment of geodesically convexity, the implicit convexity only involves analysis in Euclidean space as commonly encountered in matrix computations.  
A provable accelerated symmetric Eigensolver based on Preconditioning and Implicit Convexity (EPIC) will be proposed. Rigorous theoretical analysis of EPIC is presented and shows that the rate of convergence resembles the ``expectation'' of LOBPCG in \cite[(5.5)]{Knyazev2001}. Numerical results confirm our theoretical study.

\paragraph{Characterizations and condition number 
of strongly convex functions.}
For easy of reference, the following proposition provides the
characterizations of strongly convex functions. The proofs can be found in \cite[Chap~2.1]{Nesterov2018}.
Taking into the account of preconditioning to be discussed in this paper,  we consider a $P$-inner-product
\begin{equation} \label{eq:rmkip}
\dual{x,y}_{P}=x^{\Ttran}Py,
\end{equation}
where $P$ is a symmetric positive definite matrix. For simplicity, we use $\dual{\cdot,\cdot}$ and $\norm{\cdot}$ to denote a general inner--product and norm, which may be the $P$ inner--product and $P$ norm.

\begin{proposition} \label{prop.convexcharac} 
Suppose $\phi$ is a smooth function on a convex domain $\mathcal{Y}$,
and $0<\mu\leq L$ are positive scalars, the following three inequalities for characterizing the strongly convexity of $\phi$ are equivalent:
\begin{align}
& \frac{\mu}{2}\norm{y_{1}-y_{2}}^{2} \leq
\phi(y_{1})-\phi(y_{2})-\dual{\nabla\phi(y_{2}),y_{1}-y_{2}}\leq
\frac{L}{2}\norm{y_{1}-y_{2}}^{2}    \label{cvxmuL1} \\
& \mu\norm{y_{1}-y_{2}}\leq \norm{\nabla\phi(y_{1})-\nabla\phi(y_{2})}
\leq L\norm{y_{1}-y_{2}},  \label{cvxmuL2} \\
& \mu P \preceq \nabla^{2}\phi(y) \preceq LP, \label{cvxmuL3}
\end{align}
where $y, y_{1}$, $y_{2}\in \mathcal{Y}$.
\end{proposition}

By the convention in convex optimization \cite[P.77]{Nesterov2018},
the condition number of a strongly convex function $\phi$ is denoted by the ratio $\kappa=L/\mu$, where $L$ and $\mu$ come from \Cref{prop.convexcharac}. The condition number is closely tied to fundamental properties of algorithms. For examples, the rate of convergence of the gradient descent method and accelerated gradient descent method are bounded as $1-c\kappa$ and $1-c\kappa^{1/2}$ respectively for unconstrained convex minimization, where $c$ is some positive constant \cite[Chap~2.1]{Nesterov2018}.

\paragraph{Paper organization.}
In \Cref{sec:impconv}, we introduce the implicit convexity of the smallest eigenvalue problem by constructing an auxiliary problem on the tangent plane of an approximation of eigenvector on the $M$--sphere. 
A novel Locally Optimal scheme of Nesterov Accelerated Gradient (LONAG) flow will be proposed and analyzed in \Cref{sec:nagflow}.
In \Cref{sec:eic}, we will show that the auxiliary problem can be solved by LONAG implicitly on the $M$--sphere, which only involves some cheap operations. Such an implicit algorithm will be named as Eigensolver based on Implicit Convexity (EIC). Compared with steepest descent, the acceleration of EIC will be proved.
In \Cref{sec:epic}, a preconditioned version of EIC, which is called Eigensolver based on Preconditioning and Implicit Convexity (EPIC), will be given by involving a preconditioner $P$, which is associated with the co--preconditioner $T$ for $A$, for the auxiliary problem. Theoretical analysis show that EPIC can achieve acceleration, whose rate of convergence is faster than PSD and similar to the ``expectation'' of LOPCG.
Numerical results, including test for theoretical results and comparison with LOPCG will be given in \Cref{sec:numerics}.

\paragraph{Notation.}
We use $\dual{x,y}_{A}$ to represent the inner-product $x^{\Ttran}Ay$, where $A$ is a symmetric positive definite matrix, and $\norm{x}_{A}$ to represent its corresponding norm.
For the standard inner-product and norm in  Euclidean space, we use $\dual{\cdot,\cdot}$ and $\norm{\cdot}$ respectively.
For a symmetric positive definite pencil $(A,M)$, the notations $\lambda_{\min}(A,B)$ and $\lambda_{\max}(A,B)$ are used to represent the minimum and maximum generalized eigenvalue of $(A,B)$, respectively.
The notation $M_{1}\preceq M_{2}$ means $M_{2}-M_{1}$ is a symmetric semi-positive definite matrix.


\section{Implicit convexity of the symmetric eigenvalue problem}
\label{sec:impconv}

\subsection{The eigenvalue problem}
Suppose $A$ and $M$ are $n \times n$ symmetric positive definite matrices,
and $0<\lambda_{1}<\lambda_{2}\leq\dotsb\leq \lambda_{n}$
are eigenvalues of $(A,M)$ and $u_{1},\dotsc,u_{n}$ are the corresponding
unit eigenvectors, \ie $u_{i}^{\Ttran}Mu_{j}=0$ for $i\neq j$ and
$\norm{u_{i}}_{M}=1$.
We consider the computation of the smallest eigenvalue and 
associated eigenvector $(\lambda_1, u_1)$ of $(A,M)$:
\begin{equation} \label{oriprob}
Au_{1}=Mu_{1}\lambda_{1}.
\end{equation}
It is well-known \cite{Golub2013} that
$u_{1}$ is the unique minimizer of the Rayleigh quotient:
\begin{equation*} 
	u_{1}=\argmin_{u\neq\zero}\Rq(u)\defi
	\argmin_{u\neq\zero}\frac{\norm{u}_{A}^{2}}{\norm{u}_{M}^{2}}.
\end{equation*}

\subsection{The auxiliary problem}
In this section, we will construct an auxiliary problem of
the eigenvalue problem~\eqref{oriprob}
and then convert the eigenvalue problem~\eqref{oriprob}
into an optimization problem of a convex function over a convex domain.
Let $q$ be an approximation of the eigenvector $u_{1}$ satisfying
$q^{\Ttran}Mu_{1}>0$,\footnote{both $q$ and $-q$ are approximations
	of $u_{1}$} $\norm{q}_{M}=1$, and
\begin{equation} \label{eq:rhoq1}
\lambda_{1}\leq \rho_{q}\defi \Rq(q) < \frac{\lambda_{1}+\lambda_{2}}{2}.
\end{equation}
Let $\Sp$ be the hemisphere in $\R^{n}$:
\begin{equation*}
	\Sp\defi
	\bigg\{x\in\R^{n}\Bigm\vert \norm{x}_{M}=1,\,q^{\Ttran}Mx>0\bigg\}.
\end{equation*}
Define an $A$-spherical cap $\mathcal{X}$ of $\Sp$ as\footnote{$\mathcal{X}$ is a spherical cap defined by $A$-norm.}
\begin{equation} \label{eq:Xdef}
\mathcal{X}
= \bigg\{x\in\Sp \Bigm\vert \Rq(x)\leq\rho_{q} \bigg\} \subset \Sp.
\end{equation}
It is obvious that $\mathcal{X}$ is nonempty
since $u_{1}\in\mathcal{X}$.
Define operators $\psi\colon \Sp \mapsto \R^{n - 1}$ and $\psi^{\dagger}\colon \R^{n - 1} \mapsto \Sp$ as
\begin{equation} \label{defpsi}
	\psi (x)\defi \frac{Q^{\Ttran}Mx}{q^{\Ttran}Mx}\quad\text{and}\quad
	\psi^{\dagger}(y)\defi \frac{Qy+q}{\norm{Qy+q}_{M}}.
\end{equation}
where $Q$ is an $M$-orthogonal complement of the vector $q$, \ie
$\widetilde{Q}=[q,Q]$ is an $M$-orthogonal matrix.
The operators $\psi$ and $\psi^{\dagger}$ are well-defined,
\ie the denominators of $\psi$ and $\psi^{\dagger}$ are nonzero,
since $q$ and $Q$ are $M$-orthogonal.
Define the projected $A$-spherical cap $\mathcal{Y}$
of $\mathcal{X}$ as\footnote{$\mathcal{Y}$ is
the projection of $\mathcal{X}$.}
\begin{equation} \label{eq:Ydef}
\mathcal{Y} \defi
\bigg\{y\in\R^{n - 1}\Bigm\vert y=\psi(x),\,x\in\mathcal{X}\bigg\}.
\end{equation}
In \Cref{lemPsi},
it will be shown that $\psi^{\dagger}$ is the inverse of $\psi$ and
$\mathcal{Y} = \psi(\mathcal{X})$.

\paragraph*{Geometric interpretations.}
The tangent space of $\Sp$ at $q$ with respect to $M$-inner-product is
\begin{equation*}
\bigg\{Qy+q  \Bigm\vert y\in\R^{n - 1}\bigg\}\subset \R^{n}.
\end{equation*}
For any $x\in\Sp$,
\begin{equation} \label{eq:psiq}
Q\psi(x)+q = \frac{QQ^{\Ttran}Mx+qq^{\Ttran}Mx}{q^{\Ttran}Mx}
=\frac{x}{q^{\Ttran}Mx},
\end{equation}
where we use the fact that $\widetilde{Q} =[q,Q]$ is $M$-orthogonal.
Therefore, $Q\psi(x)+q$ is a projection of $x \in \mathcal{X}$
onto the tangent space at $q$. The operator $\psi$ maps a point $x\in\mathcal{X}$ to
the coordinates of its projection in the tangent space with the basis $Q$.
$Q\mathcal{Y}+q$ is the projection of $\mathcal{X}$ from the origin.
The relationship of $\Sp$, $\mathcal{X}$, $\mathcal{Y}$, $q$,
$u_{1}$ and $\psi(u_{1})$ is illustrated in \Cref{figq}.

\begin{figure}[t]
\centering
\includegraphics[width=0.5\textwidth]{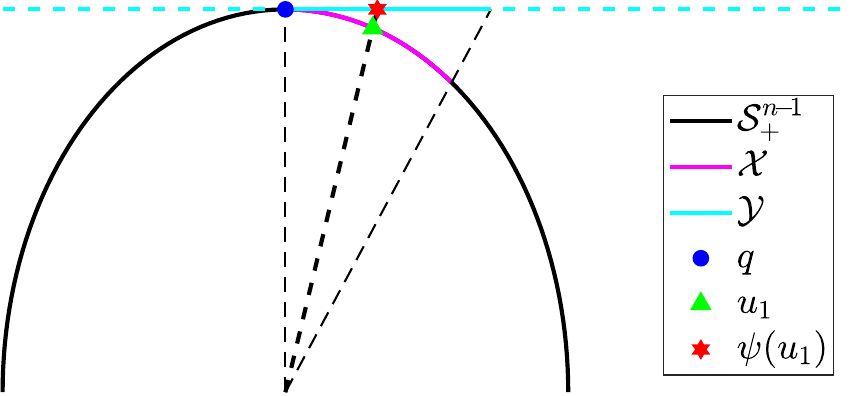}
\caption{The relationship of $\Sp$, $\mathcal{X}$, $\mathcal{Y}$,
$q$, $u_{1}$ and $\psi(u_{1})$}
\label{figq}
\end{figure}

\paragraph*{Definition of the auxiliary problem.}
Let $\phi\colon \R^{n - 1}\mapsto \R$ be defined by
\begin{equation} \label{defphi}
	\phi(y) \defi \Rq(Qy+q)
	= \frac{ y^{\Ttran}{B}y+2y^{\Ttran}b+\rho_{q}}{\norm{y}^{2}+1},
\end{equation}
where $B=Q^{\Ttran}AQ$, $b=Q^{\Ttran}Aq$ and $\rho_{q} = \Rq(q)$.
It is obvious that $\phi$ is a smooth function of $y$.
An auxiliary problem of the eigenvalue problem~\eqref{oriprob}
is defined by
\begin{equation} \label{auxprob}
	\min_{y\in\mathcal{Y}}\, \phi(y).
\end{equation}
In the rest of this section, we will show that
if $\rho_{q}$ is chosen sufficiently close to $\lambda_1$, the auxiliary function $\phi$ is strongly convex on a convex region $\mathcal{Y}$. 
Consequently, using the theory of convex optimization \cite[Thm~2.4]{Nocedal2006} and the property of $\psi^{\dagger}$ in \Cref{lemPsi}, we can conclude that the auxiliary problem \Cref{auxprob} has a unique solution $y_{*}$, and the eigenvector $u_{1}$ of the eigenvalue problem~\eqref{oriprob}
is given by $u_{1} = \psi^{\dagger}(y_{*})$.

\subsection{Properties of $\psi$ and $\psi^{\dagger}$}

We have the following lemma
on the properties of operators $\psi$ and $\psi^{\dagger}$
defined in \Cref{defpsi}.

\begin{lemma} \label{lemPsi}
For operators $\psi$ and $\psi^{\dagger}$ defined in \Cref{defpsi},
\begin{enumerate}
\item $\psi$ and $\psi^{\dagger}$ are injections,

\item $\psi^{\dagger}\bigl(\psi(x)\bigr)=x$ holds for all $x\in\Sp$,

\item $\psi\bigl(\psi^{\dagger}(y)\bigr)=y$ holds for all $y\in\R^{n - 1}$.
\end{enumerate}
Therefore, $\psi^{\dagger}$ is the inverse of $\psi$, $\mathcal{Y} = \psi(\mathcal{X})$ and $\mathcal{X} = \psi^{\dagger}(\mathcal{Y})$.
\end{lemma}
\begin{proof}
For item 1: For any $x_{1},\,x_{2}\in\Sp$, if $\psi(x_{1})=\psi(x_{2})$, we have
\begin{equation*}
Q^{\Ttran}M\Bigl(\frac{x_{1}}{q^{\Ttran}Mx_{1}}-\frac{x_{2}}{q^{\Ttran}Mx_{2}}\Bigr)=0.
\end{equation*}
By the $M$-orthogonality of $q$ and $Q$, there exists $\alpha\in\R$ such that
\begin{equation*}
\frac{x_{1}}{q^{\Ttran}Mx_{1}}-\frac{x_{2}}{q^{\Ttran}Mx_{2}}=\alpha q.
\end{equation*}
Multiplying $q^{\Ttran}M$ on the left of both sides in this equation, we know that $\alpha=0$, \ie
\begin{equation*}
x_{1}=\frac{q^{\Ttran}Mx_{1}}{q^{\Ttran}Mx_{2}} x_{2}.
\end{equation*}
Then $x_{1}=x_{2}$ is obtained by $q^{\Ttran}Mx>0$ and $\norm{x}_{M}=1$ for all $x\in\Sp$.

	For $\psi^{\dagger}$, if $\psi^{\dagger}(y_{1})=\psi^{\dagger}(y_{2})$, we have
	\begin{equation*}
		Q\Bigl(\frac{y_{1}}{\norm{Qy_{1}+q}}-\frac{y_{2}}{\norm{Qy_{2}+q}}\Bigr)=\Bigl(\frac{1}{\norm{Qy_{2}+q}}-\frac{1}{\norm{Qy_{1}+q}}\Bigr)q.
	\end{equation*}
	Using the $M$-orthogonality of $q$ and $Q$, we know that $y_{1}=y_{2}$.

	For item 2, by direct computation, for any $x\in\Sp$,
	\begin{equation*}
		\psi^{\dagger}\bigl(\psi(x)\bigr)=\dfrac{\dfrac{QQ^{\Ttran}Mx}{q^{\Ttran}Mx}+q}{\Bignorm{\dfrac{QQ^{\Ttran}Mx}{q^{\Ttran}Mx}+q}_{M}}=\dfrac{\dfrac{x-qq^{\Ttran}Mx}{q^{\Ttran}Mx}+q}{\Bignorm{\dfrac{x-qq^{\Ttran}Mx}{q^{\Ttran}Mx}+q}_{M}}=x,
	\end{equation*}
	because of $q^{\Ttran}Mx>0$ and $qq^{\Ttran}M+QQ^{\Ttran}M=I$.

	For item 3, for any $y\in\R^{n - 1}$, by $\psi^{\dagger}\bigl(\psi(x)\bigr)=x$, we know
	\begin{equation*}
		\psi^{\dagger}\Bigl(\psi\bigl(\psi^{\dagger}(y)\bigr)\Bigr)=\psi^{\dagger}(y).
	\end{equation*}
	Then $\psi\bigl(\psi^{\dagger}(y)\bigr)=y$ is obtained by $\psi$ is an injection.
\end{proof}

The following proposition establishes the connection between
the Rayleigh quotient $\Rq(\cdot)$ and the auxiliary function $\phi(\cdot)$.

\begin{proposition} \label{eqRqPhi}
Let $x\in\Sp$ and $y=\psi(x)$. Then
\begin{equation} \label{eq:rqxphiy}
\Rq(x) = \phi(y).
\end{equation}
\end{proposition}
\begin{proof} It is a direct result from
\begin{equation*}
\Rq(x) = \Rq\bigl(\psi^{\dagger}(y)\bigr) = \Rq(Qy+q) = \phi(y),
\end{equation*}
where we use \Cref{lemPsi}, the homogeneity
of the Rayleigh quotient, and \Cref{defphi}, respectively.
\end{proof}

\subsection{Convexity of $\mathcal{Y}$}
We now show that $\mathcal{Y}$ is convex.

\begin{theorem} \label{lemaux}
Under condition \Cref{eq:rhoq1},
\begin{enumerate}
\item 
$\phi(y)\leq \rho_{q}$
if and only if $y\in\mathcal{Y}$,

\item the set $\mathcal{Y}$ is convex.
\end{enumerate}
\end{theorem}
\begin{proof}
	For item 1, let $x=\psi^{\dagger}(y)$, by \Cref{eqRqPhi,eq:Xdef,eq:Ydef} we know
	\begin{equation*}
		\phi(y)\leq \rho_{q} \iff \Rq(x)\leq \rho_{q} \iff x\in\mathcal{X} \iff  y\in\mathcal{Y}.
	\end{equation*}
	For item 2, we consider an equivalent definition of $\mathcal{Y}$:
	\begin{equation*}
		\mathcal{Y}=\{y\in\R^{n - 1}\mid \phi(y)\leq \rho_{q}\}.
	\end{equation*}
	According to \cite[Lem~3.1]{Notay2002}, we know
	\begin{equation}
		\label{eigB}
		\lambda_{\min}(B) = \lambda_{\min}(Q^{\Ttran}AQ)\geq\lambda_{1}+\lambda_{2}-\rho_{q}.
	\end{equation}
	Combining \Cref{eigB,eq:rhoq1}, we have
	\begin{equation*}
		\lambda_{\min}(B)-\rho_{q}\geq \lambda_{1}+\lambda_{2}-2\rho_{q}>0,
	\end{equation*}
	which means $B-\rho_{q}I$ is a symmetric positive semi-definite matrix.
	Since
	\begin{equation}
		\label{cvxY}
		\begin{aligned}
			\phi(y)\leq \rho_{q} &\iff \frac{y^{\Ttran}By+2y^{\Ttran}b+\rho_{q}}{y^{\Ttran}y+1}\leq \rho_{q}\\
			&\iff y^{\Ttran}(B-\rho_{q}I)y+2y^{\Ttran}b\leq 0\\
			&\iff y^{\Ttran}(B-\rho_{q}I)y+2y^{\Ttran}(B-\rho_{q}I)z\leq 0\\
			&\iff (y+z)^{\Ttran}(B-\rho_{q}I)(y+z)\leq z^{\Ttran}(B-\rho_{q}I)z,
		\end{aligned}
	\end{equation}
where $z=(B-\rho_{q}I)^{-1}b$, we know that
$\mathcal{Y}$ is a closed ball with center $(-z)$ in
$(B-\rho_{q}I)$-inner-product and radius
$\bigl(z^{\Ttran}(B-\rho_{q}I)z\bigr)^{1/2}$. 
Therefore, for any $y \in \mathcal{Y}$,
\begin{equation*}
\norm{y-(-z)}_{B-\rho_{q}I} \leq \bigl(z^{\Ttran}(B-\rho_{q}I)z\bigr)^{1/2}
\end{equation*}
and $\mathcal{Y}$ is a convex set.
\end{proof}


\subsection{The convexity of $\phi$ on $\mathcal{Y}$} \label{sec:convexy}


Now let us show that the function $\phi(y)$ is
convex on $\mathcal{Y}$ by proving that $\phi(y)$ is 
a strongly convex function satisfying 
the second-order characterization \eqref{cvxmuL3}.

\begin{theorem} \label{convexP}
If the vector $q$ in the auxiliary problem \Cref{auxprob} satisfies
\begin{equation} \label{eq:positivemuP}
\lambda_{1}\leq \rho_{q} = \Rq(q)
< \lambda_{1}+\frac{\lambda_{2}-\lambda_{1}}{2+\chi_{P}},
\end{equation}
where
\begin{equation}\label{eq:chiP}
\chi_{P}=\frac{8\lambda_{2}}{\lambda_{1}}\frac{\xi_{\max}}{\xi_{\min}}
\Bigl(\frac{\lambda_{2}+\lambda_{1}}{2(\lambda_{2}-\lambda_{1})}\Bigr)^{1/2}>0,
\end{equation}
and $\xi_{\min}$ and $\xi_{\max}$ are the smallest and largest eigenvalues
of $(B,P)$, respectively.
	
Then the second-order characterization of 
the convexity of $\phi$ in the auxiliary problem \Cref{auxprob}
	\begin{equation} \label{ieqcvx2P}
	\mu_{P} P \preceq \nabla^{2}\phi(y) \preceq L_{P} P
	\end{equation}
	holds for all $y\in \mathcal{Y}$, where
	\begin{equation} \label{defCP}
	\begin{aligned}
	\mu_{P} & = 2\xi_{\min}\Bigl(1-\frac{4(\rho_{q}-\lambda_{1})}{\lambda_{2}-\lambda_{1}}\Bigr)
	\Bigl(1-\frac{\lambda_{1}}{\lambda_{2}}-\frac{(2+\chi_{P})}{\lambda_{2}}(\rho_{q}-\lambda_{1})\Bigr)>0,                                     \\
	L_{P}   & =2\xi_{\max}\Bigl(1-\frac{\lambda_{1}}{\lambda_{n}}+
 \frac{\chi_{P}}{\lambda_{2}}
 \frac{\xi_{\min}}{\xi_{\max}}(\rho_{q}-\lambda_{1})
 \Bigr).
			\end{aligned}
		\end{equation}
	\end{theorem}

	\begin{remark}
		Since $\chi_{P}>0$, condition \Cref{eq:rhoq1} holds automatically under condition \Cref{eq:positivemuP}.
	\end{remark}

	Before proving \Cref{convexP}, we first show the following two lemmas.
	The first lemma gives an upper bound for
	the angle between $x_{1}$ and $x_{2}\in\mathcal{X}$.\footnote{
	Let $\theta$ be the angle between $x_{1}$ and $x_{2}$ in $M$-inner-product, then we have
	\begin{equation*}
		\theta = \arccos \frac{\abs{x_{1}^{\Ttran}Mx_{2}}}{\norm{x_{1}}_{M}\norm{x_{2}}_{M}}.
	\end{equation*}
	Moreover, due to $x_{1}^{\Ttran}Mx_{2}>0$ and $\norm{x_{1}}_{M}=\norm{x_{2}}_{M}=1$, we have
	\begin{equation*}
		\theta = \arccos(x_{1}^{\Ttran}Mx_{2}).
	\end{equation*}
	Since $x_{1}^{\Ttran}Mx_{2} = 1-\order(\rho_{q}-\lambda_{1})$, when $\rho_{q}-\lambda_{1}$ is sufficiently small, we know $\theta=\order((\rho_{q}-\lambda_{1})^{1/2})$.
	} 
	
\begin{lemma} \label{angle} Under condition \Cref{eq:rhoq1}, for any $x_{1}$ and $x_{2}\in\mathcal{X}$,
\begin{equation*}
x_{1}^{\Ttran}Mx_{2} \geq 1- \delta_{M}>0,
\end{equation*}
where
$\delta_{M} = 2(\rho_{q}-\lambda_{1})/(\lambda_{2}-\lambda_{1})<1$.
\end{lemma}
\begin{proof}
Let us decompose $x_{j}$ on the basis of the $M$-orthonormal eigenvectors $u_{i}$: 
\begin{equation} \label{eigendecompqx}
x_{j}=\sum_{i=1}^{n}c_{i,j}u_{i}, \quad \text{and}\quad  \sum_{i=1}^{n}c_{i,j}^2=1,
\end{equation}
where $j=1$ and $2$.
Since $x_{j}\in\mathcal{X}$, we have
\begin{equation*}
\rho_{q}\geq \Rq(x_{j}) = \sum_{i=1}^{n}c_{i,j}^{2}\lambda_{i}\geq c_{1,j}^{2}\lambda_{1}+(1-c_{1,j}^{2})\lambda_{2}.
\end{equation*}
Combining it with $\rho_{q}<(\lambda_{1}+\lambda_{2})/2$ in \Cref{eq:rhoq1}, we know
\begin{equation}\label{eq:c1j}
c_{1,j}^{2}\geq \frac{\lambda_{2}-\rho_{q}}{\lambda_{2}-\lambda_{1}}>\frac{1}{2},
\end{equation}
which means for any $x\in\mathcal{X}$, $x^{\Ttran}Mu_{1}\neq0$. Since $\mathcal{X}$ is connected on the hemisphere $\Sp$ and $u_{1}\in\mathcal{X}$, we have $c_{1,j}>0$. By the Cauchy-Schwarz inequality,
\begin{equation*}
x_{1}^{\Ttran}Mx_{2}=\sum_{i=1}^{n}c_{i,1}c_{i,2}\geq c_{1,1}c_{1,2}-\frac{1}{2}\sum_{i=2}^{n}(c_{i,1}^{2}+c_{i,2}^{2})=\frac{(c_{1,1}+c_{1,2})^{2}}{2}-1 \geq 1-\frac{2(\rho_{q}-\lambda_{1})}{\lambda_{2}-\lambda_{1}},
\end{equation*}
where \eqref{eq:c1j} is used in the last inequality.
	\end{proof}
	\begin{remark}
		Due to $q\in\mathcal{X}$, we know $q^{\Ttran}Mx\geq 1-\delta_{M}$ holds for any $x\in\mathcal{X}$.
	\end{remark}
	\begin{corollary} \label{corNormY}
		For any $y\in\mathcal{Y}$, let $x=\psi^{\dagger}(y)$,
		\begin{equation*}
			\frac{y^{\Ttran}y}{1+y^{\Ttran}y}=1-(x^{\Ttran}Mq)^{2}\leq 2\delta_{M}
			\quad\text{and}\quad
			1-2\delta_{M}\leq \frac{1}{1+y^{\Ttran}y} \leq 1.
		\end{equation*}
	\end{corollary}
	
The second lemma shows the extreme eigenvalues of $\nabla\phi(y)y^{\Ttran}+y\bigl(\nabla\phi(y)\bigr)^{\Ttran}$ for any $y \in \mathcal{Y}$ in $B$ inner-product.
	
\begin{lemma} \label{esterr}
Under condition \Cref{eq:rhoq1}, for any $y\in\mathcal{Y}$,
\begin{equation}
	-\chi_{g} B \preceq
	\nabla\phi(y)y^{\Ttran}+y\bigl(\nabla\phi(y)\bigr)^{\Ttran}
	\preceq \chi_{g} B,
\end{equation}
where
\begin{equation}\label{eq:chig}
	\chi_{g} = \frac{8(\rho_{q}-\lambda_{1})}{\lambda_{1}}\Bigl(\frac{\lambda_{1}+\lambda_{2}}{2(\lambda_{2}-\lambda_{1})}\Bigr)^{1/2}.
\end{equation}
\end{lemma}
\begin{proof}
	For any $z\in\R^{n - 1}$, it is sufficient to show
	\begin{equation*}
	\abs{z^{\Ttran}\nabla\phi(y)y^{\Ttran}z+z^{\Ttran}y\bigl(\nabla\phi(y)\bigr)^{\Ttran}z}\leq
2\abs{z^{\Ttran}\nabla\phi(y)}\abs{z^{\Ttran}y}\leq {\chi_{g}} \norm{z}_{B}^{2}.
	\end{equation*}
We will prove
\begin{align}
\abs{z^{\Ttran}\nabla\phi(y)} & \leq 2\norm{z}_{B}\Bigl(\frac{\rho_{q}(\rho_{q}-\lambda_{1})}{\lambda_{1}(1+\norm{y}^{2})}\Bigr)^{1/2}, \label{esterr1} \\
			\abs{z^{\Ttran}y}             & \leq \frac{\norm{z}_{B}\norm{y}}{\sqrt{\lambda_{1}}}.
			\label{esterr2}
		\end{align}
First, consider the bound \eqref{esterr1}. The gradients of $\Rq(x)$ and $\phi(y)$  are easily computed as follows:
		\begin{equation} \label{diff}
			\begin{aligned}
				\nabla \Rq(x)  & =2\bigl(Ax-\Rq(x)Mx\bigr),                                     \\
				\nabla \phi(y) & =\frac{2}{\norm{y}^{2}+1}\bigl(By-\phi(y)y+Q^{\Ttran}Aq\bigr).
			\end{aligned}
		\end{equation}
		Let $x=\psi^{\dagger}(y)$, note that $B=Q^{\Ttran}AQ$, $\Rq(x)=\phi(y)$, $QQ^{\Ttran}M+qq^{\Ttran}M=I$ and
		\begin{equation*}
			1+\norm{y}^{2}=1+\frac{x^{\Ttran}MQQ^{\Ttran}Mx}{(q^{\Ttran}Mx)^{2}} = \frac{x^{\Ttran}Mx}{(q^{\Ttran}Mx)^{2}} = \frac{1}{(q^{\Ttran}Mx)^{2}},
		\end{equation*}
		we have
		\begin{equation} \label{gradlink}
			\begin{aligned}
				\nabla \phi(y) & = \frac{2}{\norm{y}^{2}+1}\Bigl(\frac{Q^{\Ttran}AQQ^{\Ttran}Mx}{q^{\Ttran}Mx}-\frac{\Rq(x)Q^{\Ttran}Mx}{q^{\Ttran}Mx}+Q^{\Ttran}Aq\Bigr) \\
							   & = \frac{2}{\sqrt{1+\norm{y}^{2}}}\bigl(Q^{\Ttran}A(QQ^{\Ttran}M+qq^{\Ttran}M)x-\Rq(x)Q^{\Ttran}Mx\bigr)                      \\
							   & = \frac{2}{\sqrt{1+\norm{y}^{2}}}\bigl(Q^{\Ttran}Ax-\Rq(x)Q^{\Ttran}Mx\bigr)
				= \frac{Q^{\Ttran}\nabla\Rq(x)}{\sqrt{1+\norm{y}^{2}}}.
			\end{aligned}
		\end{equation}
	Then by the Cauchy-Schwarz inequality,
		\begin{equation}
			\label{estnormy1}
			|z^{\Ttran}\nabla\phi(y)|
			= \frac{\left|(Qz)^{\Ttran}\nabla\Rq(x)\right|}{\sqrt{1+\norm{y}^{2}}}
			\leq \frac{\norm{Qz}_{A}\norm{\nabla\Rq(x)}_{A^{-1}}}{\sqrt{1+\norm{y}^{2}}}
			= \frac{\norm{z}_{B}\norm{\nabla\Rq(x)}_{A^{-1}}}{\sqrt{1+\norm{y}^{2}}}.
		\end{equation}
		Let $\rho=\Rq(x)=\phi(y)$, and assume $x=\sum_{i=1}^{n}c_{i}u_{i}$ like \Cref{eigendecompqx}, we know that
		\begin{equation*}
			\begin{aligned}
				\norm{\nabla\Rq(x)}_{A^{-1}}^{2}=4\sum_{i=1}^{n}\frac{c_{i}^{2}(\lambda_{i}-\rho)^{2}}{\lambda_{i}}.
			\end{aligned}
		\end{equation*}
		Since $x\in\mathcal{X}$, we have $\rho\leq\rho_{q}<\lambda_{2}$, then
		\begin{equation*}
			\sum_{i=1}^{n}\frac{c_{i}^{2}(\lambda_{i}-\rho)^{2}}{\lambda_{i}}
			\leq\frac{c_{1}^{2}(\rho-\lambda_{1})^{2}}{\lambda_{1}}+\sum_{i=2}^{n}c_{i}^{2}(\lambda_{i}-\rho)
			=\frac{c_{1}^{2}\rho(\rho-\lambda_{1})}{\lambda_{1}}
			\leq \frac{\rho(\rho-\lambda_{1})}{\lambda_{1}},
		\end{equation*}
		where the equation is based on the fact $\sum_{i=1}^{n}c_{i}^{2}=1$ and $\sum_{i=1}^{n}c_{i}^{2}\lambda_{i}=\rho$. Combining these two relationships, we know
		\begin{equation} \label{estnormy2}
			\norm{\nabla\Rq(x)}_{A^{-1}}^{2}\leq \frac{4\rho(\rho-\lambda_{1})}{\lambda_{1}}.
		\end{equation}
		Then \Cref{esterr1} is proved by $\rho\leq \rho_{q}$, \Cref{estnormy1,estnormy2}.
	
For the bound \Cref{esterr2}, 
		by the Cauchy-Schwarz inequality and the Courant-Fischer minimax theorem, we have
		\begin{equation*}
			\begin{aligned}
				z^{\Ttran}y
				= (Qz)^{\Ttran}MQy\leq \norm{Qz}_{A}\norm{MQy}_{A^{-1}}
				\leq \frac{\norm{z}_{B}\norm{MQy}_{M^{-1}}}{\sqrt{\lambda_{1}}}
				= \frac{\norm{z}_{B}\norm{y}}{\sqrt{\lambda_{1}}}.
			\end{aligned}
		\end{equation*}
		Then the lemma is proved by $\rho_{q}<(\lambda_{1}+\lambda_{2})/2$, \Cref{esterr1,esterr2,corNormY}.
	\end{proof}
	
	\begin{proof}[Proof for \Cref{convexP}]
	After direct computation and using \eqref{diff} again, we have
		\begin{equation*}
			\nabla^{2}\phi(y)=\frac{2}{\norm{y}^{2}+1}\Bigl(B-\phi(y)I-y\bigl(\nabla\phi(y)\bigr)^{\Ttran}-\nabla\phi(y)y^{\Ttran}\Bigr)
		\end{equation*}
		for any $y\in\R^{n - 1}$. Let $\sigma=\phi(y)\le \rho_q$, by $\lambda_{\max}(B)\le \lambda_n$,
		\begin{equation*}
			\begin{aligned}
				\max_{z\in\R^{n - 1}} & \frac{z^{\Ttran}(B-\sigma I)z}{\norm{z}_{P}^{2}}\leq \bigl(1-\sigma\lambda_{\max}^{-1}(B)\bigr)\xi_{\max}\leq \Bigl(1-\frac{\lambda_{1}}{\lambda_{n}}\Bigr)\xi_{\max},                   \\
				\min_{z\in\R^{n - 1}} & \frac{z^{\Ttran}(B-\sigma I)z}{\norm{z}_{P}^{2}}\geq \bigl(1-\sigma\lambda_{\min}^{-1}(B)\bigr)\xi_{\min}\geq \Bigl(1-\frac{\rho_{q}}{\lambda_{1}+\lambda_{2}-\rho_{q}}\Bigr)\xi_{\min},
			\end{aligned}
		\end{equation*}
		where $\lambda_{\min}^{-1}(B)$ is estimated in \Cref{eigB}.  If $\rho_q\le (\lambda_1+\lambda_2)/2$, then
		\begin{equation*}
			\frac{\rho_{q}}{\lambda_{1}+\lambda_{2}-\rho_{q}} = \frac{\lambda_{1}}{\lambda_{2}}+\frac{(\lambda_{2}+\lambda_{1})(\rho_{q}-\lambda_{1})}{\lambda_{2}(\lambda_{1}+\lambda_{2}-\rho_{q})}\leq \frac{\lambda_{1}}{\lambda_{2}}+\frac{2(\rho_{q}-\lambda_{1})}{\lambda_{2}},
		\end{equation*}
	  and $B-\phi(y)I$ can be bounded by
		\begin{equation} \label{estjP}
			\Bigl(1-\frac{\lambda_{1}}{\lambda_{2}}-\frac{2(\rho_{q}-\lambda_{1})}{\lambda_{2}}\Bigr)\xi_{\min}P\preceq B-\phi(y)I\preceq  \Bigl(1-\frac{\lambda_{1}}{\lambda_{n}}\Bigr)\xi_{\max}P.
		\end{equation}
		Using \Cref{esterr} and the definition of $\xi_{\max}$, we have
		\begin{equation} \label{estJP}
			-\chi_{g}\xi_{\max} P \preceq -\chi_{g}B \preceq
	\nabla\phi(y)y^{\Ttran}+y\bigl(\nabla\phi(y)\bigr)^{\Ttran}
	\preceq \chi_{g} B \preceq \chi_{g}\xi_{\max} P.
		\end{equation}
		Combining \Cref{estjP,estJP}, we know
		\begin{equation*}
			\frac{2\xi_{\min}}{\norm{y}^{2}+1} \Bigl(1-\frac{\lambda_{1}}{\lambda_{2}}-\frac{2(\rho_{q}-\lambda_{1})}{\lambda_{2}}-\frac{\chi_{g}\xi_{\max}}{\xi_{\min}}\Bigr) P
			\preceq \nabla^{2} \phi(y) \preceq \frac{2\xi_{\max}}{\norm{y}^{2}+1} \Bigl(1-\frac{\lambda_{1}}{\lambda_{n}}+\chi_{g}\Bigr) P
\end{equation*}
Note that $\chi_{P} = \frac{\chi_{g}\lambda_{2}\xi_{\max}}{(\rho_{q}-\lambda_{1})\xi_{\min}}$, we have
\begin{equation}\label{eq:LP}
\begin{aligned}
\frac{2\xi_{\max}}{\norm{y}^{2}+1}
\Bigl(1-\frac{\lambda_{1}}{\lambda_{n}}+\chi_{g}\Bigr)
& \leq 2\xi_{\max}\Bigl(1-\frac{\lambda_{1}}{\lambda_{n}}+\chi_{g}\Bigr) \\
& =2\xi_{\max}\Bigl(1-\frac{\lambda_{1}}{\lambda_{n}}+
\frac{\chi_{P}\xi_{\min}}{\lambda_{2}\xi_{\max}}(\rho_{q}-\lambda_{1})\Bigr) 
 \equiv L_P,
\end{aligned}
\end{equation}
and by \Cref{angle,corNormY},
\begin{equation*}
\begin{aligned}
&\frac{2\xi_{\min}}{\norm{y}^{2}+1}
\Bigl(1-\frac{\lambda_{1}}{\lambda_{2}}-
\frac{2(\rho_{q}-\lambda_{1})}{\lambda_{2}}-
\frac{\chi_{g}\xi_{\max}}{\xi_{\min}}\Bigr)\\
& \geq
2\xi_{\min}\Bigl(1-\frac{4(\rho_{q}-\lambda_{1})}{\lambda_{2}-\lambda_{1}}\Bigr)\Bigl(1-\frac{\lambda_{1}}{\lambda_{2}}-\frac{2(\rho_{q}-\lambda_{1})}{\lambda_{2}}-\frac{\chi_{g}\xi_{\max}}{\xi_{\min}}\Bigr) 
\equiv \mu_P.
\end{aligned}
\end{equation*}
which finishes the proof.
\end{proof}

	By \Cref{convexP}, we have the following two corollaries.
	\begin{corollary}
		\label{corCvxP}
		Up to the first order of $\rho_{q}-\lambda_{1}$, $\mu_{P}$ and $L_{P}$ are 
		\begin{equation*}
			\mu_{P} = 2\xi_{\min}\Bigl(1-\frac{\lambda_{1}}{\lambda_{2}}\Bigr)+\order(\rho_{q}-\lambda_{1})
			\quad\text{and}\quad
			L_{P} = 2\xi_{\max}\Bigl(1-\frac{\lambda_{1}}{\lambda_{n}}\Bigr)+\order(\rho_{q}-\lambda_{1}).
		\end{equation*}
		and the condition number $\kappa_{P}=L_{P}/\mu_{P}$ of the auxiliary function is given by 
		\begin{equation*}
			\kappa_{P} = \frac{L_{P}}{\mu_{P}} = \iota_{\xi} \frac{1-\lambda_{1}/\lambda_{n}}{1-\lambda_{1}/\lambda_{2}}+\order(\rho_{q}-\lambda_{1}),
		\end{equation*}
  where $\iota_{\xi} = \xi_{\max}/\xi_{\min}$. 
		When the standard inner--product is applied, \ie $P=I$, by the eigenvalue estimation of $B$ in \Cref{eigB}, we have 
		\begin{equation*}
			\kappa_{I} = \frac{\lambda_{n}-\lambda_{1}}{\lambda_{2}-\lambda_{1}}+\order(\rho_{q}-\lambda_{1}).
		\end{equation*}
	\end{corollary}

\begin{corollary}
	\label{revLip}
	Under condition \eqref{eq:positivemuP}, for any $y\in\mathcal{Y}$,
	\begin{equation}\label{eq:phi}
		\phi(y)-\phi(y_{*}) \leq \frac{L_{P}}{2}\norm{y-y_{*}}_{P}^{2}.
	\end{equation}
	Reversely, for any $y\in\R^{n - 1}$, if
	\begin{equation}\label{eq:y}
		\norm{y-y_{*}}_{P}^{2} \leq \frac{2(\rho_{q}-\lambda_{1})}{L_{P}},
	\end{equation}
	then $y\in\mathcal{Y}$.
\end{corollary}
\begin{proof}
	In \Cref{lemaux}, it has been proved that 
	\begin{equation*}
		\mathcal{Y}=\{y\in\R^{n - 1}\mid \lambda_{1}\leq \phi(y)\leq \rho_{q}\}.
	\end{equation*}
	By $\nabla \phi(y_{*})=\zero$, the estimation \Cref{eq:phi} is directly obtained from \Cref{convexP}. 

	Now if the condition \Cref{eq:y} is satisfied, let 
	\begin{equation*}
		\mathcal {D}\defi \biggl\{y \Bigm\vert  \norm{y-y_{*}}_{P}^{2} \leq \frac{2(\rho_{q}-\lambda_{1})}{L_{P}}\biggr\},
	\end{equation*}
	\ie $\mathcal{D}$ is the set for $y$ satisfying \Cref{eq:y}. We will show that $  \mathcal{D} \subset \mathcal{Y}$. When $\rho_{q}=\lambda_{1}$, it is obvious $\mathcal{D}=\mathcal{Y}=\{y_{*}\}$. If $\rho_{q}>\lambda_{1}$ and $\mathcal{D} \not\subset {\mathcal Y}$, then there exists $\widehat{y}_{1}\in\mathcal{D}$ but $\phi(\widehat{y}_{1})> \rho_{q}$. Note that $\phi(y_{*})=\lambda_{1}<\rho_{q}$, by the intermediate value theorem and the convexity of $\mathcal{D}$, there exists $\widehat{y}_{2}\in\mathcal{D}$ such that $\phi(\widehat{y}_{2})=\rho_{q}$ and
	\begin{equation*}
		\norm{\widehat{y}_{2}-y_{*}}_{P}^{2} < \norm{\widehat{y}_{1}-y_{*}}_{P}^{2}\leq \frac{2(\rho_{q}-\lambda_{1})}{L_{P}},
	\end{equation*}
	where the last inequalities use the fact $\widehat{y}_{1}\in\mathcal{D}$. Notice that $\widehat{y}_{2}\in \mathcal{Y}$ due to $\phi(\widehat{y}_{2})=\rho_{q}$, we can obtain 
  	\begin{equation*}
		\phi(\widehat{y}_{2})-\lambda_{1} \leq \frac{L_{P}}{2}\norm{\widehat{y}_{2}-y_{*}}_{P}^{2}<\rho_{q}-\lambda_{1}
	\end{equation*}
by \Cref{eq:phi}, which is contradicted $\phi(\widehat{y}_{2})=\rho_{q}$.
\end{proof}

\subsection{Implicit convexity of the eigenvalue problem}
Here is the main result on the implicit convexity of
the eigenvalue problem \eqref{oriprob}.

\begin{theorem} \label{thmequiv}
Suppose $\rho_{q}=\Rq(q)$ satisfies
\begin{equation} \label{ICrhoq}
\lambda_{1}\leq \rho_{q}
\leq \lambda_{1}+\frac{\lambda_{2}-\lambda_{1}}{2+\chi_{P}},
\end{equation}
where $\chi_P$ is defined in \eqref{eq:chiP}, then
\begin{enumerate}
\item the region $\mathcal{Y}$ is a convex set.
\item the auxiliary function $\phi$ is convex in $\mathcal{Y}$ with $P$ inner-product,
\item the auxiliary problem \Cref{auxprob} has a unique solution $y_{*}$,
\item the eigenvector $u_1$ of the eigenvalue problem~\eqref{oriprob}
is given by $u_{1} = \psi^{\dagger}(y_{*})$.
\end{enumerate}
\end{theorem}
\begin{proof}
	The first two items have been proved in \Cref{lemaux,convexP}.
	
	For item 3,
	according to the traditional theory about convex optimization \cite[Thm~2.4]{Nocedal2006}, the auxiliary problem \Cref{auxprob} has a unique solution
	$y_{*}=\psi(x_{*})$.
	
	For item 4, 
	let $y_{**}=\psi(u_{1})$, by $u_{1}\in\mathcal{X}$ and \Cref{eqRqPhi},
	\begin{equation*}
		\phi(y_{*})\leq \phi(y_{**})=\Rq(u_{1})\leq \Rq(x_{*})=\phi(y_{*}),
	\end{equation*}
	which means $y_{*}=y_{**}=\psi(u_{1})$. The theorem is proved by $u_{1}=\psi^{\dagger}(\psi(u_{1}))=\psi^{\dagger}(y_{*})$.
\end{proof}

\begin{remark}
Actually, only the convexity of $\phi$ depends on the inner-product. Other items, including the convexity of $\mathcal{Y}$, the existence and uniqueness of $y_{*}$, and $u_{1}=\psi^{\dagger}(y_{*})$ only require $\lambda_{1}\leq \rho_{q}<(\lambda_{1}+\lambda_{2})/2$.
\end{remark}


\section{Locally Optimal Nestrov Accelerated Gradient descent 
methods for convex optimization}
\label{sec:nagflow}

In this section, we will discuss the convex optimization problem:
\begin{equation} \label{cvxopt}
\min_{y\in\mathcal{Y}}\, \phi(y),
\end{equation}
where $\phi(y)$ is a smooth strongly convex function defined on a convex
set $\mathcal{Y}$. To save notations, we reuse $\phi$, $\mathcal{Y}$ and so on 
in \Cref{sec:impconv}. 
Following the presentation in \cite{Luo2021}, we will review some results about NAG methods with a dynamical system analogy first proposed in \cite{Su2014}. Then, we propose a new discretization scheme and analyze its rate of convergence.


\subsection{NAG methods with a dynamical system analogy}

\paragraph{Dynamical system.} 
Consider the following first-order dynamical system of $\bigl(y(t),s(t)\bigr)$:
\begin{subnumcases}{\label{NAGflow}}
\frac{\de y(t)}{\de t} =s(t)-y(t), \label{NAGflowa}  \\
\frac{\de s(t)}{\de t} =y(t)-s(t)-\frac{1}{\mu}\nabla \phi(y(t)), \label{NAGflowb}
\end{subnumcases}
with initial conditions $y(0)=y_{0}$ and $s(0)=s_{0}$,
where $t>0$, $\phi$ and $\mu$ satisfy \Cref{cvxmuL1}.
To establish the connection between solution of
the optimization~\eqref{cvxopt} and the dynamical system \eqref{NAGflow},
let us consider the following so-called Lyapunov function:
\begin{equation} \label{eq:lyapunov}
\Ly(t) = \phi\bigl(y(t)\bigr)-\phi(y_{*})+\frac{\mu}{2}\norm{s(t)-y_{*}}^{2}\geq 0,
\end{equation}
where $y_{*}$ is the {unique} minimizer of \Cref{cvxopt}.
It is shown in \cite[Lem~2]{Luo2021} that
the Lyapunov function $\Ly(t)$ exponentially decays:
\begin{equation} \label{eq:lyafunbd}
\Ly(t)\leq e^{-t}\Ly(0).
\end{equation}
Note that $\Ly(t)\geq 0$, combining \Cref{eq:lyafunbd,eq:lyapunov}, we know 
\begin{equation*}
	\lim_{t\to\infty}\phi\bigl(y(t)\bigr)-\phi(y_{*})\leq \lim_{t\to\infty}\Ly(t) = 0.
\end{equation*}
Then we know  
$y(t) \to y_{*}$ as $t\to\infty$ since $y_{*}$ is the unique minimizer.\footnote{The algorithm can 
be understood as ODEs, which converge
to an equilibrium in the continuous time domain. This allows
a unique view and understanding of the discrete iterative process.}

\paragraph{Discrete schemes.}
There are a number of discrete schemes for the dynamical
system \eqref{NAGflow}  \cite{Su2014,Odonoghue2015,Shi2021,Muehlebach2021,Luo2021}.
To balance the efficiency and stability,
we focus on the following corrected semi-implicit scheme
\cite[(96--97)]{Luo2021}.
Given the initial $(s_{0}, y_{0}) \in (\mathcal{Y},\mathcal{Y})$,
$\mu$ and $L$ as defined in \Cref{cvxmuL1}, step--size $\tau >0$,
the corrected semi-implicit scheme generates the iterates
\begin{equation*}
(\overline{y}_k, s_{k+1}, y_{k+1})
\end{equation*}
for $k = 0,1,2,\dotsc$, by the recursions
\begin{subnumcases}{\label{GC}}
\frac{\overline{y}_{k}-y_{k}}{\tau }=s_{k}-\overline{y}_{k}, \label{GCa} \\
\frac{s_{k+1}-s_{k}}{\tau }=(\overline{y}_{k}-s_{k})
-\frac{1}{\mu}\nabla \phi(\overline{y}_{k}), \label{GCb} \\
\text{update $y_{k+1}$ satisfying } \phi(y_{k+1})\leq
\phi(\overline{y}_{k})-\frac{1}{2L}\norm{\nabla\phi(\overline{y}_{k})}^{2}.
\label{GCc}
\end{subnumcases}
When the first order characterization \Cref{cvxmuL1} holds globally, a popular choice for $y_{k+1}$ in the step \eqref{GCc} is a gradient step \cite[(2.2.19)]{Nesterov2018}, \ie
\begin{equation*}
y_{k+1} = \overline{y}_{k}-\frac{1}{L}\nabla\phi(\overline{y}_{k}).
\end{equation*}
There are also some other implementations of \eqref{GCc}
under the globally first order characterization \cite{Luo2021}.
However, when \Cref{cvxmuL1} only holds locally, the discussion
is relatively few.
We will give a new implementation
of \eqref{GCc} and analyze its convergence.

\medskip

The following theorem from \cite[Thm~7]{Luo2021} proves the convergence rate 
of the scheme~\eqref{GC}.

\begin{theorem}[\!\!{\cite[Thm~7]{Luo2021}}] \label{GCprop}
Let $y_{0}\in\mathcal{Y}$ and $s_{0}\in\mathcal{Y}$ be the
initials and $\tau >0$ be the step--size. Assume that
\begin{itemize} 
\item the step--size $\tau $ satisfies $0<\tau \leq \kappa^{-1/2}$, where
$\kappa=L/\mu$, and $\mu$ and $L$ as defined in \Cref{cvxmuL1}.
\item all iterates $(\overline{y}_k, y_{k+1}, s_{k+1})$ from
the corrected semi-implicit scheme \eqref{GC} lie in $\mathcal{Y}$.
\end{itemize} 
Then
\begin{equation} \label{expdecay}
\Ly_{k+1}\leq (1-\tau )\Ly_{k},
\end{equation}
where
\begin{equation} \label{Ly}
\Ly_{k}=\phi(y_{k})-\phi(y_{*})+\frac{\mu}{2}\norm{s_{k}-y_{*}}^{2}.
\end{equation}
\end{theorem}

By the inequality \eqref{expdecay}, the convergence of discrete Lyapounov function  implies the convergence of $y_{k}$ to $y_{*}$. Taking the optimal step--size as $\tau =\kappa^{-1/2}$, we can achieve the acceleration by improving the rate of convergence from $1-2(\kappa+1)^{-1}$ of the gradient method \cite[Thm~2.1.15]{Nesterov2018} to $1-\kappa^{-1/2}$.

\subsection{LONAG scheme and convergence analysis}

There are two issues with the corrected semi-implicit scheme \eqref{GC}: (1) no guarantee
for the monotonic declination of $\phi(y_{k})$,
which actually may fluctuate \cite{Park2021}, and (2) 
assumptions about all iterates
$(\overline{y}_{k}, s_{k+1}, y_{k+1})$ in $\mathcal{Y}$ 
are necessary for the convergence.
In this section, we propose a new scheme
to guarantee $\phi(y_{k})$ decreasing monotonically and
analyze its convergence with conditions only about
the initial values $(s_{0},y_{0})$.

\paragraph{LONAG scheme.} 
Given the initial $(s_{0}, y_{0}) \in (\mathcal{Y},\mathcal{Y})$,
$\mu$ and $L$ as defined in \Cref{cvxmuL1}, step--size $\tau >0$,
we propose to replace the update~\eqref{GCc} with
a locally optimal correction, and generate
\[
(\overline{y}_k, s_{k+1}, y_{k+1})
\]
for $k = 0,1,2,\dotsc,$ by the recursions
\begin{subnumcases}{\label{LO}}
\frac{\overline{y}_{k}-y_{k}}{\tau }=s_{k}-\overline{y}_{k}, \label{LOa} \\
\frac{s_{k+1}-s_{k}}{\tau }=(\overline{y}_{k}-s_{k})-\frac{1}{\mu}\nabla
\phi(\overline{y}_{k}), \label{LOb} \\
y_{k+1}=\argmin_{y\in \spa{y_{k},\overline{y}_{k},
\nabla\phi(\overline{y}_{k})}}\phi(y). \label{LOc}
\end{subnumcases}
The update~\eqref{LOc} for $y_{k+1}$ is
inspired by the LOPCG \cite{Knyazev2001}. Since the scheme \eqref{LO} 
is a combination of the locally optimal step and the NAG flow, 
we name it Locally Optimal Nesterov Accelerated Gradient (LONAG).


\paragraph{Monotonicity of LONAG.} 
Due to the locally optimal step \Cref{LOc}, the monotonically declination
of function values $\phi(y_{k})$ can be obtained directly.

\begin{proposition} \label{Lyaux0}
The iterates $\{y_{k+1}\}$ from the LONAG \eqref{LO} satisfies
\begin{equation*}
\phi(y_{k+1})\leq \phi(y_{k})\leq  \dotsb \leq \phi(y_{0}).
\end{equation*}
\end{proposition}
\begin{proof}
It is a direct result the fact
$y_{k+1}\in\spa{y_{k},\overline{y}_{k},\nabla\phi(\overline{y}_{k})}$
from the locally optimal step \eqref{LOc}. 
\end{proof}

\begin{remark}
As a consequence of \Cref{Lyaux0}, when the level set property
\begin{equation*}
\bigl\{y\mid \phi(y)\leq \phi(y_{0})\bigr\}   \subset \mathcal{Y}
\end{equation*}
holds with proper choice of the initial $y_0$, 
the last step \eqref{LOc} is equivalent to
\begin{equation*}
y_{k+1}=\argmin_{y\in \mathcal{Y}\cap \spa{y_{k},\overline{y}_{k},
\nabla\phi(\overline{y}_{k})}}\phi(y).
\end{equation*}
\end{remark}


\paragraph{Containment and convergence.} 
For the convergence of LONAG, we would like to use the convergence of
the corrected semi-implicit scheme \eqref{GC} shown in \Cref{GCprop}. The main difficulty is that there is no prior assumption the containment of iterates $(\overline{y}_{k}, s_{k+1}, y_{k+1})$. However, once the convergence of the discrete Lyapounov function $\Ly_{k}$ is proved as \cref{expdecay}, we know that $y_{k}$ and $s_{k}$ can not be too far from the minimizer $y_{*}$. 
A lucky fact is that we can prove these two properties, \ie containment and convergence, recursively when initials $(y_{0},s_{0})$ and step--size $\tau $ are properly selected.

\begin{theorem} \label{thmInvarset}
Assume that
\begin{itemize} 
\item the initial $(s_0, y_{0})$ satisfies
\begin{equation} \label{invarset}
s_{0}\in\mathcal{B}_{R_{1}} 
\quad \mbox{and} \quad
\bigl\{y\mid \phi(y)\leq \phi(y_{0})\bigr\} \subset
\mathcal{B}_{R_{1}}\subset  \mathcal{B}_{R_{2}}  \subset \mathcal{Y},
\end{equation}
where $\mathcal{B}_{R}$ is a closed ball with center $y_{*}$,
the unique optimizer of \eqref{cvxopt}, and radius $R$:
\begin{equation*}
\mathcal{B}_{R}\defi \bigl\{ y\mid \norm{y-y_{*}}\leq R  \bigr\},
\end{equation*}
and 
\begin{equation} \label{defR}
R_{1}=\Bigl(\frac{2\Ly_{0}}{\mu}\Bigr)^{1/2}\quad\text{and}\quad
R_{2}=\max\bigg\{2R_{1},(1+\tau\kappa)R_{1}\bigg\},
\end{equation}
and $\Ly_{0}$ is the initial discrete Lyapunov function,
\ie $\Ly_{0}=\phi(y_{0})-\phi(y_{*})+\frac{\mu}{2}\norm{s_{0}-y_{*}}^{2}$.
\item the step--size $\tau $ satisfies $0<\tau \leq \kappa^{-1/2}$, 
where $\kappa=L/\mu$, and $\mu$ and $L$ as defined in \Cref{cvxmuL1}.
\end{itemize} 
Then the iterates $(\overline{y}_k, s_{k+1}, y_{k+1})$  with $k > 0$
generated by the LONAG \eqref{LO} satisfy
\begin{enumerate}[(a)]
\item $\overline{y}_{k}\in\mathcal{B}_{R_{1}}$
$y_{k+1}\in\mathcal{B}_{R_{1}}$.

\item $y_{k+1}$ satisfies the sufficient declination property
\eqref{GCc} of the corrected semi-implicit scheme \eqref{GC}.
\item $\Ly_{k+1}\leq (1-\tau )\Ly_{k}$, where $\Ly_{k}$ is defined as $\Ly_{k}=\phi(y_{k})-\phi(y_{*})+\frac{\mu}{2}\norm{s_{k}-y_{*}}^{2}$.
\item $s_{k+1}\in\mathcal{B}_{R_{1}}$.

\end{enumerate}
\end{theorem}
\begin{proof}
The conclusions will be proved recursively. Let us assume that both $y_{k}$ and $s_{k}$ are in $\mathcal{B}_{R_{1}}$, which are satisfied if $k=0$,
then by \eqref{LOa}, we have
\[
\overline{y}_{k}= \frac{y_k}{1 + \tau } + \frac{\tau s_k}{1 + \tau }
 \in\mathcal{B}_{R_{1}}.
\]
 Rearranging \eqref{LOb}, we have
\begin{equation*}
s_{k+1} = (1-\tau )s_{k}+\tau \overline{y}_{k}
-\frac{\tau }{\mu}\nabla\phi (\overline{y}_{k}) .
\end{equation*}
Then we have
\begin{equation*}
\begin{aligned}
\norm{s_{k+1}-y_{*}}
& = \norm{(1-\tau )s_{k}+\tau \overline{y}_{k}
-\frac{\tau }{\mu}\nabla\phi (\overline{y}_{k}) - y_*} \\
& = \norm{(1-\tau )(s_{k}-y_{*})
           +\tau (\overline{y}_{k}-y_{*})
           -\frac{\tau }{\mu}\nabla\phi (\overline{y}_{k})} \\
&\leq
(1-\tau )\norm{s_{k}-y_{*}}+\tau \norm{\overline{y}_{k}-y_{*}} +
\frac{\tau }{\mu}\norm{\nabla\phi(\overline{y}_{k})}\\
&\leq (1-\tau )R_{1}+\tau R_{1}+\tau \kappa R_{1}\leq  R_{2},
\end{aligned}
\end{equation*}
where for the second inequality, we used the inequality
\Cref{normGradYBar} and $\kappa=L/\mu$, and for
the last inequality we used \Cref{defR}.
Therefore, $s_{k+1}\in\mathcal{B}_{R_{2}} \subset \mathcal{Y}$.

Now according to the monotone declination \Cref{Lyaux0}, we know that
\begin{equation}
\phi(y_{k+1})\leq  \phi(y_{k})\le \phi(y_{0}).
\end{equation}
Therefore by containment property \Cref{invarset},  ${y}_{k+1}\in\mathcal{B}_{R_{1}}$ is true.

We can also
show that $y_{k+1}$ of LONAG also satisfies
the sufficient declination of $\phi(y_{k})$ in
the corrected semi-implicit scheme \eqref{GCc}.
Let
\begin{equation*}
\widetilde{y}_{k}
= \overline{y}_{k}-\frac{1}{L}\nabla\phi(\overline{y}_{k})\in
\spa{y_{k},\overline{y}_{k},\nabla\phi(\overline{y}_{k})}.
\end{equation*}
First, by  \Cref{cvxmuL2} and $\nabla\phi(y_{*})=\zero$, we know
\begin{equation} \label{normGradYBar}
\norm{\nabla\phi(\overline{y}_{k})}
= \norm{\nabla\phi(\overline{y}_{k})-\nabla\phi(y_{*})}
\leq L\norm{\overline{y}_{k}-y_{*}}\leq LR_{1},
\end{equation}
where we used the fact that $\norm{\overline{y}_{k}-y_{*}}\leq R_{1}$ since $\overline{y}_{k} \in\mathcal{B}_{R_{1}}$.
By \Cref{normGradYBar}, we have
\begin{equation*}
\norm{\widetilde{y}_{k}-y_{*}}
\leq \norm{\overline{y}_{k}-y_{*}}+
\frac{1}{L}\norm{\nabla\phi(\overline{y}_{k})}
\leq 2R_{1}\leq R_{2},
\end{equation*}
which means
\begin{equation*}
\widetilde{y}_{k}
\in \mathcal{B}_{2R_{1}}
\subset \mathcal{B}_{R_{2}}\subset \mathcal{Y}.
\end{equation*}
Now using \eqref{cvxmuL1}, we have
\begin{equation} \label{suffY1}
\begin{aligned}
\phi(y_{k+1}) &\leq \phi(\widetilde{y}_{k})  \leq \phi(\overline{y}_{k})+ \dual{\nabla \phi(\overline{y}_{k}), \widetilde{y}_{k}-\overline{y}_{k}}
+\frac{L}{2}\norm{\widetilde{y}_{k}-\overline{y}_{k}}^{2}\\
&=\phi(\overline{y}_{k})-\frac{1}{L}\norm{\nabla\phi(\overline{y}_{k})}^{2}
+\frac{1}{2L}\norm{\nabla\phi(\overline{y}_{k})}^{2}\\
&=\phi(\overline{y}_{k})-\frac{1}{2L}\norm{\nabla\phi(\overline{y}_{k})}^{2},
\end{aligned}
\end{equation}
where the first inequality comes from the locally optimal step \eqref{LOc}
and the second inequality comes from the first-order
characterization \Cref{cvxmuL1}.

Thus we have proved that
$
\overline{y}_k$, $y_{k+1} \in \mathcal{B}_{R_{1}} \subset \mathcal{Y}, $ $
s_{k+1} \in \mathcal{B}_{R_{2}}\subset\mathcal{Y},
$
and
$y_{k+1}$ of LONAG satisfies the sufficient declination of
$\phi(y_{1})$ in \eqref{GCc}.

	Now we are ready to use
the convergence of the corrected semi-implicit scheme
in \Cref{GCprop} to obtain
\begin{equation}
	\label{eq:Lkdecreasing} 
\begin{aligned}
\Ly_{k+1}  = \phi(y_{k+1})-\phi(y_{*})+\frac{\mu}{2}\norm{s_{k+1}-y_{*}}^{2}   \leq (1-\tau ) \Ly_{k}\leq \dotsb \leq (1-\tau )^{k+1}\Ly_{0}.
\end{aligned}
\end{equation}
Note that
$\phi(y_{k+1})-\phi(y_{*})\geq 0$ always holds since $y_{*}$ is the minimizer of $\phi$, therefore, we have
\[
\frac{\mu}{2}\norm{s_{k+1}-y_{*}}^{2}
\leq \phi(y_{k+1})-\phi(y_{*})+\frac{\mu}{2}\norm{s_{k+1}-y_{*}}^{2} =\Ly_{k+1}
\leq \Ly_{0}
\]
which implies that $s_{k+1}\in\mathcal{B}_{R_{1}}.$ 
This completes the proof. 
\end{proof}

\begin{remark}
	Let us point out the major difference between \Cref{GCprop,thmInvarset}. 
	In \Cref{GCprop}, the iterates $(\overline{y}_k, y_{k+1}, s_{k+1})$ 
	are assumed to locate in $\mathcal{Y}$, which is the domain where 
	$\phi$ satisfies the first order characterization \Cref{cvxmuL1}. 
	Such an assumption is very general in convex optimization, 
	since the objective function is globally convex.
	However, for the auxiliary problem \Cref{auxprob}, the objective function is locally convex with respect to the choice of 
	$q$, an approximate eigenvector of $u_1$. There is no prior assumption about the locations of iterates $(\overline{y}_k, y_{k+1}, s_{k+1})$. 
	In this scenario, we need to prove the containment 
	like \Cref{invarset}, which was inspired by 
	the work of Park, Salgado and Wise \cite{Park2021} 
	on preconditioned Nesterov accelerated gradient method for solving semilinear PDEs.
\end{remark}

 By the monotonically declination in \Cref{Lyaux0} and convergence in \Cref{thmInvarset}, we have the following results on the convergence of the LONAG scheme.

\begin{corollary} \label{Lyaux}
With the assumptions of \Cref{thmInvarset},
the sequence $\{y_k\}$ generated by the LONAG~\eqref{LO} satisfies that
\begin{equation} \label{Lyaux1}  
\phi(y_{k})\leq \phi(y_{k-1})\leq \dotsb \leq \phi(y_{0})
\end{equation} 
and 
\begin{equation} \label{Lyaux2}  
\phi(y_{k})-\phi(y_{*})\leq (1-\tau )^{k}\Ly_{0},
\end{equation}
where $0<\tau \leq \kappa^{-1/2}$,
$y_{*}$ is the minimizer of \Cref{cvxopt} and $\Ly_{0}=\phi(y_{0})-\phi(y_{*})+\frac{\mu}{2}\norm{s_{0}-y_{*}}^{2}$. 
\end{corollary}

It is clear that the LONAG, similar to the convergence rate of the corrected semi-implicit scheme in \Cref{GCprop}, can also achieve the acceleration by improving the rate of convergence to $1-\kappa^{-1/2}$.


\section{EIC: a symmetric Eigensolver based on Implicit Convexity}
\label{sec:eic}

In this section, we will propose an algorithm for solving the original eigenvalue problem~\eqref{oriprob} 
by transforming the LONAG \eqref{LO} for the auxiliary function \Cref{auxprob} on $\mathcal{Y}$ into  $\mathcal{X}$. The new algorithm is called Eigensolver based on Implicit Convexity, EIC for short. In addition, we will discuss the convergence of the algorithm and the need of preconditioning. 

\subsection{EIC}
\label{sec:eiconX}

Let us return to the auxiliary problem \Cref{auxprob}
\begin{equation*} \tag{\ref{auxprob}}
	\min_{y\in\mathcal{Y}}\, \phi(y)
	\defi \frac{ y^{\Ttran}{B}y+2y^{\Ttran}b+\rho_{q}}{\norm{y}^{2}+1}.
\end{equation*}
As shown in \Cref{thmequiv}, the auxiliary problem \Cref{auxprob} is a
locally convex optimization problem, and we can apply the LONAG~\eqref{LO} 
for solving \eqref{auxprob}. 

\paragraph{LONAG for the auxiliary problem on $\mathcal{Y}$.} 
With initial $(s_{0},y_{0})\in (\mathcal{Y},\mathcal{Y})$, 
the LONAG scheme~\eqref{LO} generates the iterates 
$(\overline{y}_k, s_{k+1}, y_{k+1})$ by the following recursions:
\begin{subnumcases}{\label{algoY}}
\overline{y}_{k}=\dfrac{y_{k}+\tau s_{k}}{1+\tau }, \label{algoY1} \\
s_{k+1}=(1-\tau )s_{k}+\tau \overline{y}_{k}-
\dfrac{\tau }{\mu} \nabla\phi(\overline{y}_{k}), \label{algoY3} \\
y_{k+1} =
\argmin\limits_{y\in \mathcal{Y}\cap\spa{y_{k},\overline{y}_{k},
\nabla\phi(\overline{y}_{k})}} \phi(y), \label{algoY4}
\end{subnumcases}
where the stepsize $\tau $ satisfies 
$0<\tau \leq \kappa^{-1/2}$, 
$\kappa = L/\mu$, $\mu$ and $L$ are convexity parameters of $\phi(y)$ 
defined in \Cref{convexP} with $P=I$.

\paragraph{EIC = LONAG for the auxiliary problem on $\mathcal{X}$.} 
In \eqref{algoY}, we solve
the auxiliary problem \Cref{auxprob} on
$\mathcal{Y}$, and assume that $\widetilde{Q} =  [q, Q]$ is
explicit available. This is not practical since the use of $Q$ is too expensive.
To circumvent $Q$, 
we propose a scheme by
transforming the computation on $\mathcal{Y}$ into $\mathcal{X}$ without using $Q$.
To do so, for $k \geq 0$, denote
\begin{equation} \label{eq:ytox}
\overline{x}_{k}=\psi^{\dagger}(\overline{y}_{k}),\
z_{k}=\psi^{\dagger}(s_{k}),\
x_{k}=\psi^{\dagger}(y_{k}), \
\end{equation}
where the operator $\psi^{\dagger}$ is defined as in \eqref{defpsi}.
Since $\psi^{\dagger}\colon \mathcal{Y}\mapsto\mathcal{X}$
(\Cref{lemPsi}), it is clear that
\begin{equation*}
z_{k}, x_{k}, \overline{x}_{k} \in \mathcal{X}
\end{equation*}

The following lemma shows that the explicit reference of $Q$ can be avoided after
applying $\psi^{\dagger}$ due to the equation 
$\mathcal{Y}=\psi(\mathcal{X})$ established in \Cref{lemPsi}.

\begin{lemma} \label{linearPsi}
For any $x\in\R^{n}$,
\begin{equation*}
\psi^{\dagger}(Q^{\Ttran}Mx)
=\frac{q+(I-qq^{\Ttran}M)x}{\bignorm{q+(I-qq^{\Ttran}M)x}_{M}},
\end{equation*}
where $[q, Q]$ is an $M$-orthogonal matrix.
\end{lemma} 
\begin{proof}
It is a direct result from the $M$-orthogonality of $[q,Q]$ 
and the definition of $\psi^{\dagger}$ in \Cref{defpsi}.
\end{proof}

Now let us reveal the expressions of 
$\overline{x}_{k}, z_{k+1},  x_{k+1}$ without explicit reference of $Q$.
First, for $\overline{x}_{k}$, 
by \Cref{linearPsi} and the definition of $\psi^{\dagger}$ in \Cref{defpsi},
\begin{equation}
	\label{updw} 
	\begin{aligned}
		\overline{x}_{k}
& =\psi^{\dagger}(\overline{y}_{k})  
=\psi^{\dagger}\Bigl(\frac{y_{k}+\tau s_{k}}{1+\tau }\Bigr)  
=\psi^{\dagger}\Bigl(\frac{\psi(x_{k})+\tau \psi(z_{k})}{1+\tau }\Bigr)  \\
& =\psi^{\dagger}\biggl(Q^{\Ttran}M\Bigl(\frac{x_{k}}{(1+\tau )(q^{\Ttran}Mx_{k})}+\frac{\tau z_{k}}{(1+\tau )(q^{\Ttran}Mz_{k})}\Bigr)\biggr) 
=\dfrac{1}{\eta_{1}}\biggl( \dfrac{x_{k}}{q^{\Ttran}Mx_{k}}+
\dfrac{\tau z_{k}}{q^{\Ttran}Mz_{k}} \biggr), 
	\end{aligned}
\end{equation}
where
$\eta_{1}$ is a scaling factor such that $\norm{\overline{x}_{k}}_{M} = 1$.

Next, consider $z_{k+1}$.
According to the definition of $\psi$ in \Cref{defpsi} and
the gradient of $\phi$ in \Cref{gradlink},
and using \Cref{linearPsi} and $z_{k+1}=\psi^{\dagger}(s_{k+1})$, we have
\begin{equation}
	\label{updz}
	\begin{aligned}
		z_{k+1}
& = \psi^{\dagger}(s_{k+1})  
= \psi^{\dagger}\Bigl((1-\tau )s_{k}+\tau \overline{y}_{k}-(\tau /\mu)\nabla\phi(\overline{y}_{k})  \Bigr)  \\
& = \psi^{\dagger}\Bigl(
(1-\tau )\psi(z_{k})+ \tau \psi(\overline{x}_{k})
-\frac{\tau (q^{\Ttran}M\overline{x}_{k})}{\mu}Q^{\Ttran}r_{k} \Bigr) 
 \\
& = \psi^{\dagger}\biggl(
Q^{\Ttran}M\Bigl(\frac{(1-\tau )z_{k}}{q^{\Ttran}Mz_{k}}
+\frac{\tau \overline{x}_{k}}{q^{\Ttran}M\overline{x}_{k}}
-\frac{\tau (q^{\Ttran}M\overline{x}_{k})M^{-1}r_{k}}
{\mu}\Bigr) \biggr)  \\
& = \dfrac{1}{\eta_{2}} \biggl(
\dfrac{(1-\tau )z_{k}}{q^{\Ttran}Mz_{k}}+
\dfrac{\tau \overline{x}_{k}}{q^{\Ttran}M\overline{x}_{k}}-
\dfrac{\tau (q^{\Ttran}M\overline{x}_{k})(I-qq^{\Ttran}M)M^{-1}r_{k}}{\mu}
\biggr),  
	\end{aligned}
\end{equation}
where $r_{k}=\nabla\Rq(\overline{x}_{k})
=2\bigl(A\overline{x}_{k}-\Rq(\overline{x}_{k})M\overline{x}_{k}\bigr)$,
$\eta_{2}$ is a scaling factor such that $\norm{z_{k+1}}_{M} = 1$.


Finally for $x_{k+1}$, 
consider the local optimization problem \eqref{algoY4}:
\begin{equation*}
	y_{k+1} = \argmin_{y\in \mathcal{V}_{\mathcal{Y}}}\phi(y),
\end{equation*}
where
$\mathcal{V}_{\mathcal{Y}}
=\mathcal{Y}\cap\spa{y_{k},\overline{y}_{k},\nabla\phi(\overline{y}_{k})}$.
Let
\begin{equation*}
	\mathcal{V}_{\mathcal{X}}
	=\mathcal{X}\cap\spa{q,x_{k},\overline{x}_{k},M^{-1}
	\nabla\Rq(\overline{x}_{k})}.
\end{equation*}
Then from the definition of $\psi$ in \Cref{defpsi} and
the gradient of $\phi$ in \Cref{gradlink}, we know
\begin{equation*}
	\psi(\mathcal{V}_{\mathcal{X}})\subset \mathcal{V}_{\mathcal{Y}}
	\quad\text{and}\quad
	\psi^{\dagger}(\mathcal{V}_{\mathcal{Y}})\subset \mathcal{V}_{\mathcal{X}},
\end{equation*}
which means
$\psi(\mathcal{V}_{\mathcal{X}})=\mathcal{V}_{\mathcal{Y}}.$
Now consider the following local optimization problem
on $\mathcal{X}$:
\begin{equation*}
	x_{*} = \argmin_{x\in \mathcal{V}_{\mathcal{X}}}\Rq(x)
\end{equation*}
By the minimization property of $x_{*}$ and $y_{k+1}$, and \Cref{eqRqPhi}, we have
\begin{equation*}
	\Rq(x_{*})\leq \Rq\bigl(\psi^{\dagger}(y_{k+1})\bigr)
	=\phi(y_{k+1})\leq \phi\bigl(\psi(x_{*})\bigr)=\Rq(x_{*}).
\end{equation*}
Due to the uniqueness of $x_{*}$ and $y_{k+1}$,
we obtain
\begin{equation} \label{updx}
	x_{k+1} = \psi^{\dagger}(y_{k+1}) = x_{*}
	= \argmin_{x\in \mathcal{V}_{\mathcal{X}}}\Rq(x).
\end{equation}

\paragraph{EIC.} 
Combining \Cref{updw,updz,updx}, we derive an equivalent
iteration of \eqref{algoY} with all computations are on $\mathcal{X}$.
The recursions \Cref{updw,updz,updx} with initial values 
$x_{0}, z_0\in\mathcal{X}$ are called  Eigensolver based on Implicit Convexity (EIC).

\subsection{Convergence of EIC}

\begin{theorem} \label{thmEIC}
Assume that 
\begin{itemize} 
\item the stepsize $0<\tau \leq \kappa^{-1/2}$,
where $\kappa=L/\mu$, $\mu=\mu_{P}$ and $L=L_{P}$ 
are defined in \Cref{convexP} with $P=I$,

\item the initial vector $x_{0} \in \mathcal{X}$ of
the EIC~\Cref{updw,updz,updx} is chosen such that
\begin{equation} \label{defrho0}
0 \leq \Rq(x_0) -\lambda_{1} \leq
\frac{1}{\max\bigl\{8\kappa,2\kappa(1+\tau \kappa)^{2}\bigr\}}
(\rho_{q} - \lambda_{1}),
\end{equation}
where $\rho_{q}=\Rq(q)$ satisfies the condition~\Cref{ICrhoq},
and $q$ is the vector in \Cref{eq:rhoq1} to define
the auxiliary problem \Cref{auxprob}.

\item the initial vector $z_0 = x_{0}$.\footnote{We enforce
$z_{0}=x_{0}$ and $s_{0}=y_{0}$ in LONAG to simplify the proof.}
\end{itemize} 
Then the Rayleigh quotient sequence of $x_{k}$
generated by EIC~\Cref{updw,updz,updx} satisfy
\begin{equation} \label{thmRq} 
\Rq(x_{k})\leq \Rq(x_{k-1})\leq \dotsb \leq \Rq(x_{0}), 
\end{equation} 
and 
\begin{equation} \label{thmEk}
\Rq(x_{k})-\lambda_{1}\leq  
2(1-\tau )^{k} \bigl(\Rq(x_{0})-\lambda_{1}\bigr).
\end{equation} 
\end{theorem}
\begin{proof}
The monotonicity of the Rayleigh quotient sequence $\Rq(x_{k})$
in \eqref{thmRq} is a direct consequence from the local optimization problem \Cref{updx}.

For the convergence of the Rayleigh quotient
sequence $\Rq(x_k)$ in \Cref{thmEk},
since EIC \Cref{updw,updz,updx} is
equivalent to applying the LONAG \eqref{algoY} for the auxiliary
problem \eqref{algoY}, the convergence of EIC can be concluded 
by verifying that the assumption \eqref{invarset} of \Cref{Lyaux}
is satisfied if the initial vector $x_0$ is chosen to satisfy \eqref{defrho0}.
Therefore, for the rest of the proof, we need to show that
\begin{enumerate}[(i)]
\item
if the initial vector $x_0$ of EIC \Cref{updw,updz,updx} is
chosen to satisfy \eqref{defrho0},
then the assumption \eqref{invarset} of \Cref{thmInvarset} holds, \ie
\begin{equation} \tag{\ref{invarset}}
\bigl\{y\mid \phi(y)\leq \phi(y_{0})\bigr\} \subset
\mathcal{B}_{R_{1}}\subset  \mathcal{B}_{R_{2}} \subset \mathcal{Y}
\end{equation}
where
$R_{1}=(2\Ly_{0}/\mu)^{1/2}$
and
$R_{2}=\max\{2R_{1},(1+\tau \kappa)R_{1}\}$.

\item From the declination \eqref{eq:Lkdecreasing}
of the discrete Lyapounov function $\mathcal{L}_{k}$ of \Cref{thmInvarset},
we show the convergence of the Rayleigh quotient sequence $\{\Rq(x_k)\}$
as in \eqref{thmEk}.
\end{enumerate}

For the item (i), by \Cref{eqRqPhi}, we know that
$\Rq(x_{0}) = \phi(y_{0})$. Therefore we need to show that if
\begin{equation} \label{defrho0Y}
\phi(y_{0})-\lambda_{1} \leq
\frac{1}{\max\bigl\{8\kappa,2\kappa(1+\tau \kappa)^{2}\bigr\}}
\bigl(\rho_{q}-\lambda_{1}\bigr),
\end{equation}
then the assumption \eqref{invarset} of \Cref{thmInvarset} holds.
Let us first show that
\begin{equation} \label{invarset1}
\bigl\{y\mid \phi(y)\leq \phi(y_{0})\bigr\} \subset \mathcal{B}_{R_{1}}.
\end{equation}
In fact, by \Cref{lemaux} and $\Rq(x_{0})\leq \rho_{q}$, we have
\begin{equation*}
\bigl\{y\mid \phi(y)\leq \phi(y_{0})\bigr\} =
\bigl\{y\mid \phi(y)\leq \Rq(x_0) \bigr\} \subset
\bigl\{y\mid \phi(y)\leq \rho_{q} \bigr\} = \mathcal{Y}.
\end{equation*}
Furthermore,  for any $y$ satisfying $\phi(y)\leq \phi(y_{0})$,
by the convexity of $\phi$ on $\mathcal{Y}$,
the first-order characterization \eqref{cvxmuL1} and
$\nabla \phi(y_{*})=\zero$, we have
\begin{equation} \label{normyystar}
\norm{y-y_{*}}^{2}\leq \frac{2}{\mu}\bigl(\phi(y)-\phi(y_{*})\bigr)
\leq \frac{2}{\mu}\bigl(\phi(y_{0})-\phi(y_{*})\bigr)
\leq \frac{2\Ly_{0}}{\mu},
\end{equation}
which means $y\in\mathcal{B}_{R_{1}}$. Therefore, 
\Cref{invarset1} is proved.

For the other two relationships, \ie
\begin{equation} \label{eq:R1inR2}
\mathcal{B}_{R_{1}}\subset  \mathcal{B}_{R_{2}}\quad \text{and}\quad \mathcal{B}_{R_{2}}\subset \mathcal{Y},
\end{equation}
the first one comes from $R_{1}\leq R_{2}$. For the second one, note that 
\begin{equation*}
	\Ly_{0} = \phi(y_{0})-\phi(y_{*})+\frac{\mu}{2}\norm{y_{0}-y_{*}}^{2}
	\leq 2\bigl(\phi(y_{0})-\phi(y_{*})\bigr)
	= 2(\Rq(x_{0})-\lambda_{1}),
\end{equation*}
we can obtain $\mathcal{B}_{R_{2}}\subset \mathcal{Y}$ by 
\begin{equation*}
R_{2}^{2} 
= \max\{8,2(1+\tau \kappa)^{2}\}\frac{\Ly_{0}}{\mu} 
\leq \max\{16,4(1+\tau \kappa)^{2}\}\frac{\Rq(x_{0})-\lambda_{1}}{\mu} 
\leq \frac{2(\rho_{q}-\lambda_{1})}{L}
\end{equation*}
and \Cref{revLip}.
Combining \Cref{invarset1,eq:R1inR2}, we conclude \Cref{invarset}.


For item (ii).
Since we have proved that $\rho_{0}$ satisfies \Cref{defrho0},
the convergence of LONAG in \Cref{Lyaux2} hold, \ie
\begin{equation*}
	\phi(y_{k})-\lambda_1\leq (1-\tau )^{k} \Ly_{0},
\end{equation*}
Combining it with \Cref{eqRqPhi}, we have
\begin{equation*}
\Rq(x_{k})-\lambda_{1}
=\phi(y_{k})- \phi(y_{*}) 
\leq  2(1-\tau )^{k} (\Rq(x_{0})-\lambda_{1}),
\end{equation*}
which is the result \Cref{thmEk}.
\end{proof}

Combining the convergence analysis of EIC in \Cref{thmEIC} with the estimation for the condition number of the auxiliary function in \Cref{corCvxP}, neglecting the term with $\rho_{q}-\lambda_{1}$, the rate of convergence for EIC is 
\begin{equation}
	\label{rateEIC}
	\Rq(x_{k})-\lambda_{1}\leq 2\biggl(1-\Bigl(\frac{\lambda_{2}-\lambda_{1}}{\lambda_{n}-\lambda_{1}}\Bigr)^{1/2}\biggr)^{k}\bigl(\Rq(x_{0})-\lambda_{1}\bigr). 
\end{equation}
Compared with the convergence rate of the steepest descent method \cite[Thm~2.1]{Knyazev1991}:
\begin{equation*}
	\tan\Theta(x_{k},u_{1})\leq \Bigl(1-\frac{\lambda_{2}-\lambda_{1}}{\lambda_{n}-\lambda_{1}}\Bigr)^{k}\tan\Theta(x_{0},u_{1}),
\end{equation*}
EIC achieves the acceleration by improving the exponent of $\frac{\lambda_{2}-\lambda_{1}}{\lambda_{n}-\lambda_{1}}$ from $1$ to $1/2$.

On the other hand, the bound \Cref{rateEIC} is not satisfactory in practices. When the ratio of the spectral spread $(\lambda_{n}-\lambda_{1})$ and the spectral gap $(\lambda_{2}-\lambda_{1})$ is large, for example the ratio of discrete Laplacian operator $\Delta^{h}$ is $h^{-2}$, where $h$ is mesh size, the rate of convergence is also close to $1$, such as $1-h$ for $\Delta^{h}$, which leads to slow convergence of the EIC.

Meanwhile, we observe that in \Cref{corCvxP}, the condition number $\kappa_{P}$ will be improved to $\lambda_{2}/(\lambda_{2}-\lambda_{1})$ when the matrix $P$ is a good spectral approximation of $B$ such that the ratio $\iota_{\xi}$ defined in \eqref{convexP} is close to $1$, which leads to the fast convergence of EIC.
In the next section, we will plugin the preconditioning technique to EIC to improve the condition number $\kappa_{P}$ by using a proper chosen preconditioner $P$.  
The resulting algorithm is called Eigensolver based on Preconditioning and Implicit Convexity (EPIC).


\section{EPIC}
\label{sec:epic}

\subsection{EPIC = Preconditioned EIC }
\paragraph{LONAG in $P$ inner--product on $\mathcal{Y}$.} 
Let us again start with the auxiliary problem \Cref{auxprob}
on $\mathcal{Y}$. Instead of the standard inner--product, we use the $P$ inner--product now. Given initials $s_0$ and $y_0$, then the LONAG scheme \eqref{algoY} in $P$ inner--product is as follows:
\begin{subnumcases}{\label{LOP}}
\overline{y}_{k}=\dfrac{y_{k}+\tau_{P} s_{k}}{1+\tau_{P}},
\label{LOPa} \\
s_{k+1}=(1-\tau_{P})s_{k}+\tau_{P}\overline{y}_{k}-
\dfrac{\tau_{P}}{\mu} P^{-1}\nabla\phi(\overline{y}_{k}) , \label{LOPb} \\
y_{k+1}=\argmin_{y\in \mathcal{Y}\cap\spa{y_{k},\overline{y}_{k},
P^{-1}\nabla\phi(\overline{y}_{k})}}\phi(y), \label{LOPc}
\end{subnumcases}
where $P$ is a symmetric positive definite matrix. 
In some literatures \cite{Park2021}, such a strategy is called preconditioning since the level sets of the objective $\phi$ look more circular when some good $P$ is applied. Throughout this section, we will also call $P$ as a preconditioner and the scheme \Cref{LOP} as a preconditioned LONAG.

\paragraph{Preconditioned LONAG on $\mathcal{X}$.} 
Like \Cref{sec:eiconX}, we would like to compute the
preconditioned LONAG flow \eqref{LOP} on $\mathcal{X}$.
To do so, for $k \geq 0$, let
\begin{equation*}
\overline{x}_{k}=\psi^{\dagger}(\overline{y}_{k}), \
z_{k}=\psi^{\dagger}(s_{k}),\
x_{k}=\psi^{\dagger}(y_{k}).
\end{equation*}

By \Cref{updw}, $\overline{x}_{k}$ can be updated as
\begin{equation} \label{updxP}
\overline{x}_{k}=\dfrac{1}{\eta_{1}}
\Bigl( \dfrac{x_{k}}{q^{\Ttran}Mx_{k}}+\dfrac{\tau_{P} z_{k}}{q^{\Ttran}Mz_{k}} \Bigr),
\end{equation}
where $\eta_{1}$ is a scaling factor such that
$\norm{\overline{x}_{k}}_{M}=1$.

For $z_{k+1}$, similar to \Cref{updz}, we know
\begin{align} 
z_{k+1} & = \dfrac{1}{\eta_{2}} \Bigl(
\dfrac{(1-\tau_{P})z_{k}}{q^{\Ttran}Mz_{k}}+
\dfrac{\tau_{P} \overline{x}_{k}}{q^{\Ttran}M\overline{x}_{k}}-
\dfrac{\tau_{P}(q^{\Ttran}M\overline{x}_{k})(I-qq^{\Ttran}M)QP^{-1}Q^{\Ttran}r_{k}}{\mu} \Bigr) \nonumber \\
& = \dfrac{1}{\eta_{2}} \Bigl(
\dfrac{(1-\tau_{P})z_{k}}{q^{\Ttran}Mz_{k}}+
\dfrac{\tau_{P} \overline{x}_{k}}{q^{\Ttran}M\overline{x}_{k}}-
\dfrac{\tau_{P}(q^{\Ttran}M\overline{x}_{k})QP^{-1}Q^{\Ttran}r_{k}}{\mu} \Bigr), \label{updzP}
\end{align}
where
\begin{equation} \label{defrkP}
r_{k} = \nabla\Rq(\overline{x}_{k}) =  2 \bigl(A\overline{x}_{k}-\Rq(\overline{x}_{k})M\overline{x}_{k}\bigr),
\end{equation}
and $\eta_{2}$ is a scaling factor such that $\norm{z_{k+1}}_{M} = 1$.

Note that the computation for the vector $z_{k+1}$ of \Cref{updzP}
is unattainable due to the vector
$QP^{-1}Q^{\Ttran}r_{k}$ involving the matrix $Q$.
To circumvent $Q$, we introduce a symmetric positive
definite co--preconditioner $T$ of $P$, where $T\in\R^{n \times n}$,
and enforce the form of $P$ as
\begin{equation} \label{eq:PQprecond} 
P=Q^{\Ttran}TQ.	
\end{equation}

The following lemma shows that the vector $QP^{-1}Q^{\Ttran}r_{k}$ can be computed with explicit 
reference of $Q$.

\begin{lemma} \label{defPT}
Suppose $T$ is symmetric positive definite and $P=Q^{\Ttran}TQ$.
Then for any $z\in\R^{n}$,
\begin{equation} \label{eq:qpqz} 
QP^{-1}Q^{\Ttran}z = \proj T^{-1}z,
\end{equation}
where $\proj$ is a complementation of the oblique projector
$\widetilde{q}q^{\Ttran}M/(q^{\Ttran}M\widetilde{q})$ defined as: 
\begin{equation} \label{eq:compobproj} 
\proj = I-\frac{\widetilde{q}q^{\Ttran}M}{q^{\Ttran}M\widetilde{q}},
\end{equation}
and $\widetilde{q} = T^{-1}Mq$. 
\end{lemma}
\begin{proof}
Since $[q,Q]$ is an $M$-orthogonal matrix, it is sufficient to prove
\begin{align}
q^{\Ttran}M\bigl(QP^{-1}Q^{\Ttran}z\bigr) & = q^{\Ttran}M\proj T^{-1}z,\label{defPT1} \\
Q^{\Ttran}M\bigl(QP^{-1}Q^{\Ttran}z\bigr) & = Q^{\Ttran}M\proj T^{-1}z.\label{defPT2}
\end{align}

	For \Cref{defPT1}, the left side is zero due to the $M$-orthogonality of $[q,Q]$, and the right side is also zero due to $\proj^{\Ttran}Mq=\zero$.

	For \Cref{defPT2}, multiplying $P$ on both sides, it is sufficient to prove
	\begin{equation*}
		Q^{\Ttran}z = PQ^{\Ttran}M\Bigl(I-\frac{\widetilde{q}q^{\Ttran}M}{q^{\Ttran}M\widetilde{q}}\Bigr)T^{-1}z.
	\end{equation*}
	By $P=Q^{\Ttran}TQ$ and $\proj^{\Ttran}Mq=\zero$, we have
	\begin{equation*}
		\begin{aligned}
			PQ^{\Ttran}M\proj T^{-1}z
			 & =Q^{\Ttran}TQQ^{\Ttran}M\proj T^{-1}z
			=  Q^{\Ttran}T(I-qq^{\Ttran}M)\proj T^{-1}z                                                                 \\
			 & =Q^{\Ttran}z-\frac{z^{\Ttran}\widetilde{q}}{q^{\Ttran}M\widetilde{q}} Q^{\Ttran}TT^{-1}Mq = Q^{\Ttran}z,
		\end{aligned}
	\end{equation*}
	which means \Cref{defPT2} holds. Then the lemma is proved by \Cref{defPT1,defPT2}.
\end{proof}

By \Cref{defPT}, the updating formula \Cref{updzP} can be rewritten as
\begin{equation} \label{eq:EPICz}
z_{k+1} = \frac{1}{\eta_{2}}\Bigl(
	\dfrac{(1-\tau_{P})z_{k}}{q^{\Ttran}Mz_{k}}+
	\dfrac{\tau_{P} \overline{x}_{k}}{q^{\Ttran}M\overline{x}_{k}}-
	\dfrac{\tau_{P}(q^{\Ttran}M\overline{x}_{k})\widetilde{r}_{k}}{\mu}\Bigr),
\end{equation}
where
\begin{equation*}
\widetilde{r}_{k} = QP^{-1}Q^{\Ttran}r_{k} = \proj T^{-1}r_{k}, 
\end{equation*}
$\proj$ is defined in \eqref{eq:compobproj},  
and $\eta_{2}$ is a scaling factor such that $\norm{z_{k+1}}_{M}=1$.

Finally, for the vector $x_{k+1}$, 
let us consider the local optimization problem
\begin{equation*}
	y_{k+1} = \argmin_{y\in\mathcal{V}_{\mathcal{Y}}}\phi(y),
\end{equation*}
where
\begin{equation*}
\mathcal{V}_{\mathcal{Y}}
=\mathcal{Y}\cap\spa{y_{k},\overline{y}_{k},P^{-1}\nabla\phi(\overline{y}_{k})}.
\end{equation*}
According \Cref{gradlink,defrkP},
\begin{equation*}
	P^{-1}\nabla\phi(\overline{y}_{k})
	= \frac{P^{-1}Q^{\Ttran}\nabla\Rq(\overline{x}_{k})}{\sqrt{1+\norm{\overline{y}_{k}}^{2}}} = \frac{P^{-1}Q^{\Ttran}r_{k}}{\sqrt{1+\norm{\overline{y}_{k}}^{2}}}.
\end{equation*}
Combining this equation with \Cref{defPT,defpsi}, we have
\begin{equation*}
	\mathrm{span}\bigl\{q,\psi^{\dagger}\bigl(P^{-1}\nabla\phi(\overline{y}_{k})\bigr)\bigr\} = \spa{q,QP^{-1}\nabla\phi(\overline{y}_{k})} = \spa{q,\widetilde{r}_{k}}.
\end{equation*}
With same arguments of \Cref{updx}, let
\begin{equation*}
	\mathcal{V}_{\mathcal{X}}
	= \mathcal{X}\cap\spa{q,x_{k},\overline{x}_{k},\widetilde{r}_{k}}.
\end{equation*}
We know $\psi(\mathcal{V}_{\mathcal{X}})=\mathcal{V}_{\mathcal{Y}}$ and
the expression of $x_{k+1}$ is
\begin{equation} \label{eq:EPICx}
x_{k+1}=\psi^{\dagger}(y_{k+1})
= \argmin_{x\in \mathcal{V}_{\mathcal{X}}} \Rq(x).
\end{equation}

\subsection{EPIC pseudocode}
Combining \Cref{updxP,updzP,eq:EPICx}, we have a preconditioned LONAG on $\mathcal{X}$ as \Cref{algoP}, which is called Eigensolver based on Preconditioning and Implicit Convexity, EPIC in short.

\medskip

\begin{algorithm2e}[H]  
	\caption{EPIC} \label{algoP}
	\KwIn{Matrices $A, M$, a vector $q$, a preconditioner $T$, the initial vector $x_{0}$, and parameters $0<\mu\leq L$.}

	Compute $T\widetilde{q} = Mq$ for $\widetilde{q}$ and $\tau_{P}=\sqrt{\mu/L}$\;

	Set $k=0$, $z_{0}=x_{0}$ and $\alpha_{0}=\gamma_{0}=q^{\Ttran}Mx_{0}$\;

	\Repeat{Convergence}{

	Compute $\overline{x}_{k}=\dfrac{x_{k}}{\alpha_{k}}+\dfrac{\tau_{P} z_{k}}{\gamma_{k}}$\;

	Normalize $\overline{x}_{k}$ by $\overline{x}_{k}=\overline{x}_{k}/\norm{\overline{x}_{k}}_{M}$\;

	Compute $\beta_{k}=q^{\Ttran}M\overline{x}_{k}$, $\rho_{k}=\Rq(\overline{x}_{k})$ and $r_{k}=2(A\overline{x}_{k}-\rho_{k}M\overline{x}_{k})$\;

	Compute $\widetilde{r}_{k} = \proj T^{-1}r_{k}$, where
        $\proj = I-\frac{\widetilde{q}q^{\Ttran}M}{q^{\Ttran}M\widetilde{q}}$\;

	Compute $z_{k+1}=\dfrac{(1-\tau_{P})z_{k}}{\gamma_{k}}+\dfrac{\tau_{P} \overline{x}_{k}}{\beta_{k}}-\dfrac{\tau_{P}\beta_{k}\widetilde{r}_{k}}{\mu}$\;

	Normalize $z_{k+1}$ by $z_{k+1}=z_{k+1}/\norm{z_{k+1}}_{M}$\;

	Compute $\gamma_{k+1} = q^{\Ttran}Mz_{k+1}$\;

	Solve a local optimization problem $x_{k+1} = \argmin\limits_{x\in \mathcal{X}\cap\spa{q,x_{k},\overline{x}_{k},\widetilde{r}_{k}}}\Rq(x)$\;

	Compute $\alpha_{k+1}=q^{\Ttran}Mx_{k+1}$\;

	Set $k=k+1$\;

	}  
\end{algorithm2e}


\noindent 

\begin{remark}
	According to Stewart's analysis of oblique projectors
in \cite{Stewart2011}, the cancellation may happen during computing
the complementation $\proj$. A remedy is to repeat the process,
which is called recomplementation.
\end{remark}


\paragraph{Complexity, EPIC vs LOPCG.} 
In each iteration of the EPIC, one matrix-vector multiplication of $A$ for computing the residual vector $r_{k}=2(A\overline{x}_{k}-\rho_{k}M\overline{x}_{k})$, 
one preconditioned linear system $T^{-1}r_{k}$, and one 
Rayleigh--Ritz procedure are needed. The difference is that LOPCG compute the Rayleigh--Ritz procedure in a three--dimensional subspace while EPIC in a four--dimensional subspace.

If taking the matrix-vector multiplication of $M$ into account, 
since $Mq$ can be computed in advance, we only need to compute 
two $M$-orthogonalization, \ie $\overline{x}_{k}$ and $z_{k+1}$, 
and one $M$ matrix-vector multiplication for residual vector $r_{k}$, where LOPCG only needs one matrix-vector multiplication. Since the major cost comes from the preconditioned linear systems and matrix-vector multiplications of $A$,  the cost of EPIC and LOPCG are the essentially same.

\subsection{Convergence analysis of EPIC}

Like \Cref{thmEIC}, we can establish the convergence of EPIC by applying the preconditioned LONAG for the auxiliary problem.
\begin{theorem} \label{thmEPIC}
Assume that
\begin{itemize} 
\item the step--size $\tau_{P}$ satisfies 
\begin{equation} \label{defAlphaP} 
0<\tau_{P}\leq \kappa_{P}^{-1/2}, 
\end{equation}
where 
$\kappa_{P}=L_{P}/\mu_{P}$, $\mu_{P}$ and $L_{P}$ are defined in 
\Cref{convexP},

\item the initial vector $x_{0} \in \mathcal{X}$ of EPIC is chosen such that
\begin{equation*} 
0 \leq \Rq(x_0) -\lambda_{1} \leq
\frac{1}{\max\bigl\{8\kappa_P,2\kappa_P(1+\tau_{P}\kappa_P)^{2}\bigr\}}
(\rho_{q} - \lambda_{1}),
\end{equation*}
where $\rho_{q}=\Rq(q)$ satisfies the condition~\Cref{ICrhoq},
and $q$ is the vector in \Cref{eq:rhoq1} to define
the auxiliary problem \Cref{auxprob}.
\item $z_{0}=x_{0}$ in EPIC (\Cref{algoP}).  
\end{itemize} 
Then the Rayleigh quotient sequence of $x_{k}$
generated by EPIC (\cref{algoP}) satisfy
\begin{equation}
\Rq(x_{k})\leq \Rq(x_{k-1})\leq \dotsb \leq \Rq(x_{0}), \label{eq:Rqprecond}
\end{equation} 
and
\begin{equation} \label{eq:rateprecond}
\Rq(x_{k})-\lambda_{1}\leq  2(1-\tau_{P})^{k} \bigl(\Rq(x_{0})-\lambda_{1}\bigr). 
\end{equation}
\end{theorem}

The proof of \Cref{thmEPIC} is analogous to the proof of \Cref{thmEIC}. The only difference is replacing the standard inner--product by $P$ inner--product. 
Similar to the discussion for the convergence of EIC in \Cref{rateEIC}, neglecting the term with $\rho_{q}-\lambda_{1}$, the rate of convergence for EPIC is 
\begin{equation}
	\label{rateEPIC}
	\Rq(x_{k})-\lambda_{1}\leq 2(1-\sqrt{\eta_{\xi}})^{k}\bigl(\Rq(x_{0})-\lambda_{1}\bigr), 
\end{equation}
where 
\begin{equation*}
    \eta_{\xi} = \frac{1-\lambda_{1}/\lambda_{n}}{\iota_{\xi}(1-\lambda_{1}/\lambda_{2})} 
\quad \text{and} \quad 
\iota_{\xi} = \frac{\xi_{\max}}{\xi_{\min}}
= \frac{\lambda_{\max}(B,P)}{\lambda_{\min}(B,P)}.    
\end{equation*}

Clearly, the bound \eqref{rateEPIC} is better than the following sharp estimation for the preconditioned inverse iteration in \cite{Argentati2017}\footnote{The result in \cite{Argentati2017} is slightly different. In their result, there is no $\lambda_{n}$ term in $\eta_{\xi}$.}
\begin{equation*}
	\frac{\Rq(x_{k+1})-\lambda_{1}}{\lambda_{2}-\Rq(x_{k+1})}\leq (1-\eta_{\xi})^{2}\frac{\Rq(x_{k})-\lambda_{1}}{\lambda_{2}-\Rq(x_{k})},
\end{equation*}
since the exponent of $\eta_{\xi}$ is $1/2$ rather than $1$. For LOPCG, Knyazev gave the following expected rate of convergence in \cite{Knyazev2001}:
\begin{equation}
	\label{rateLOPCG}
	\frac{\Rq(x_{k+1})-\lambda_{1}}{\lambda_{2}-\Rq(x_{k+1})}\leq \Bigl(1-\frac{2\sqrt{\eta_{\xi}}}{1+\sqrt{\eta_{\xi}}}\Bigr)^{2}\frac{\Rq(x_{k})-\lambda_{1}}{\lambda_{2}-\Rq(x_{k})}.
\end{equation}

To the best of our knowledge, a complete proof of upper bound~\eqref{rateLOPCG} is elusive so far. Recently, a provable accelerated eigensolver with preconditioning named Riemannian Acceleration with Preconditioning (RAP) is proposed in \cite{Shao2023b}. The RAP achieves an acceleration similar to \Cref{rateEPIC}, but the analysis is different. For RAP, the geodesical convexity is well--studied, but the preconditioning is complicate since the operations are on manifold. Some extra terms, besides $\eta_{\xi}$, about the preconditioner $T$ are involved for the theoretical gaurantee of acceleration. For EPIC, due to the subtle structure of the implicit convexity and transformation between eigenvalue problem and auxiliary problem, the preconditioning is very natural. For the convergence rate, up to the first order of $\rho_{q}-\lambda_{1}$, we only need the traditional term $\eta_{\xi}$. 

To end this section, let us discuss how to quantify the quality of preconditioner $P$, which is equivalent co--preconditioner $T$. 
First, from the practical viewpoint, the linear system $Tx=b$ 
should be easy to solve. From the theoretical viewpoint, based on the rate of convergence for EPIC in \Cref{rateEPIC}, the ratio $\iota_{\xi}$ should be close to $1$. Since 
	\begin{equation}
		\label{compareIota}
		\iota_{\xi} := \frac{\xi_{\max}}{\xi_{\min}} 
= \frac{\lambda_{\max}(B,P)}{\lambda_{\min}(B,P)} 
\leq \frac{\nu_{\max}}{\nu_{\min}} 
= \frac{\lambda_{\max}(A,T)}{\lambda_{\min}(A,T)} := \iota_{\nu},
	\end{equation}
	we can select the co--preconditioner $T$ as a good spectral approximation of $A$, \ie $\iota_{\nu}$ is close to $1$.
Finally, as a by-product, a good preconditioner $P$ (therefore, the co--preconditioner $T$) enlarges the permissible region for 
the choice of initial vector $q$, since the requirement of $\rho_{q}$ in
\Cref{eq:positivemuP} is 
\begin{equation*}
\lambda_{1}\leq \rho_{q} \leq
\lambda_{1}+\frac{\lambda_{2}-\lambda_{1}}{2+\chi_{P}},
\end{equation*}
where
\[ 
\chi_{P}=\frac{8\lambda_{2}}{\lambda_{1}} 
\cdot \iota_{\xi} \cdot
\Bigl(\frac{\lambda_{2}+\lambda_{1}}{2(\lambda_{2}-\lambda_{1})}\Bigr)^{1/2}.
\] 
According to \Cref{compareIota}, we know $\iota_{\xi}\leq \iota_{\nu}$. 
Thus, when $T$ is a good preconditioner for $A$, 
\ie $\iota_{\nu}$ is close to $1$, 
the parameter $\chi_{P}$ will be significantly contracted,
and the permissible region for $\rho_{q}$ is enlarged.

\section{Numerical experiments}
\label{sec:numerics} 

In this section, numerical results are presented to
support our theoretical analysis above.
In the first example, we will look into the sharpness of the exponent $-1/2$ in $\kappa_{P}^{-1/2}$ from \Cref{thmEPIC}.
In the second example, we select some matrices and matrix pencils
from SuiteSparse Matrix Collection to compare the performance and behavior of
EPIC with LOPCG, a popular preconditioned eigensolver with momentum.

\subsection{Tests for sharpness of exponent}

Following the setting in \cite[Sec~6]{Knyazev2001}, let
\begin{equation*}
A = \text{Diag}(\lambda_{1}, \lambda_{2}, \dotsc, \lambda_{n})
\quad \text{and}\quad
M=I,
\end{equation*}
where $\lambda_{i} = \omega^{i-1}$ for some $\omega>1$. Then the ratio
of spectral spread and spectral gap is
\begin{equation*}
\frac{\lambda_{n}-\lambda_{1}}{\lambda_{2}-\lambda_{1}}
= \frac{\omega^{n-1}-1}{\omega-1}\geq \omega^{n-2}.
\end{equation*}
When $\omega>1$, the ratio will grow exponentially
and the eigenvalue problem is ill-conditioned.

The co--preconditioner $T$ is constructed as
\begin{equation*}
T = A^{1/2}S^{-1}D^{-1}SA^{1/2},
\end{equation*}
where $S$ and $S^{-1}$ are the discrete sine transformation matrix 
and its inverse, which can be implemented by Matlab built-in 
function $\texttt{dst}$ and $\texttt{idst}$ respectively, and
\begin{equation*}
D = \mathtt{Diag\bigl(logspace(0,log10(\iota_{\nu}),n)\bigr)},
\end{equation*}
where $\iota_{\nu}>1$ is a parameter. 
The preconditioner $P$ is given by $P=Q^{\Ttran}TQ$, where $[q,Q]$ is an $M$--orthogonal matrix, and $q$ is an approximation of the eigenvector $u_{1}$ to be determined later.

According to the Courant-Fischer minimax theorem, we know
\begin{equation*}
\begin{aligned}
\nu_{\min} &\defi  \lambda_{\min}(A,T) = 1 \leq \lambda_{\min}(Q^{\Ttran}AQ,Q^{\Ttran}TQ) = \lambda_{\min}(B,P) \defi \xi_{\min},\\
\nu_{\max} &\defi  \lambda_{\max}(A,T) = \iota_{\nu}  \geq \lambda_{\max}(Q^{\Ttran}AQ,Q^{\Ttran}TQ) = \lambda_{\max}(B,P)\defi \xi_{\max}.
\end{aligned}
\end{equation*}
As shown \Cref{corCvxP},
up to the first order of $\rho_{q} - \lambda_{1}$, the 
parameters $\mu_{P}$ and $L_{P}$ for the convexity of
the function $\phi$ in \Cref{convexP} are
\begin{equation} \label{defmuLNE}
\begin{aligned}
\mu_{P} & = 2\xi_{\min}\Bigl(1-\frac{\lambda_{1}}{\lambda_{2}}\Bigr)
\geq \frac{2(\omega-1)}{\omega} +\order(\rho_{q}-\lambda_{1}), \\
L_{P} & = 2\xi_{\max}\Bigl(1-\frac{\lambda_{1}}{\lambda_{n}}\Bigr)
	\leq \frac{2\iota_{\nu}(\omega^{n-1}-1)}{\omega^{n-1}}+\order(\rho_{q}-\lambda_{1}).
\end{aligned}
\end{equation}
Then, the condition number of the auxiliary function $\phi$ in the $P$-inner product is
\begin{equation*}
\kappa_{P} = \frac{L_{P}}{\mu_{P}}
\leq  \frac{\omega^{n-1}-1}{\omega^{n-2}(\omega-1)} \iota_{\nu}+\order(\rho_{q}-\lambda_{1}).
\end{equation*}
For fixed $\omega$ and $n$, neglecting the high order terms, the condition number $\kappa_{P}$ is bounded by
\begin{equation} \label{NEkappa1}
\iota_{\nu} \leq \kappa_{P} \leq \Bigl(1+\frac{1}{\omega}\Bigr)\iota_{\nu}.
\end{equation}
Therefore, we can modify $\iota_{\nu}$ to adjust the ratio of the largest and smallest generalized eigenvalue of the matrix pencil $(A,T)$, \ie $\iota_{\nu}$, for different condition number $\kappa_{P}$ of the convex function $\phi$.

Let $\epsilon_{k}=f(x_{k})-\lambda_{1}$, where $\{x_{k}\}$ are iteration points. By the rate of convergence for EPIC in \Cref{eq:rateprecond}, 
with same initial value $x_{0}$ and stopping criteria $\epsilon_{k}\leq \epsilon_{*}$, we have
\begin{equation} \label{NEkappa2}
	\ln\Bigl(\frac{\epsilon_{*}}{2\epsilon_{0}}\Bigr)\leq m_{P}\ln(1-\tau_{P}),
\end{equation}
where $m_{P}$ is the iteration number of EPIC until convergence. 
When the step--size is chosen as $\tau_{P}=\kappa_{P}^{-1/2}$, by the first order expansion of $\ln(1-\tau_{P})$ and \Cref{NEkappa1}, we obtain
\begin{equation} \label{NEkappa3}
-\ln(1-\tau_{P}) \approx \tau_{P}
= \kappa_{P}^{-1/2}\approx \iota_{\nu}^{-1/2}.
\end{equation}
Combining \Cref{NEkappa2,NEkappa3}, we know the relationship between $\iota_{\nu}$ and $m_{P}$ should be
\begin{equation} \label{NEkappa}
m_{P} \leq \dfrac{\ln\bigl(\frac{\epsilon_{*}}{2\epsilon_{0}}\bigr)}{-\ln(1-\tau_{P})} \leq C\iota_{\nu}^{1/2}\ln\Bigl(\frac{\epsilon_{*}}{2\epsilon_{0}}\Bigr),
\end{equation}
where $C$ is an absolute constant from the approximation \Cref{NEkappa3}. With \Cref{NEkappa}, we could expect the iteration number of EPIC
will increase in the order $\iota_{\nu}^{1/2}$.

For numerical examples, we set $n=512$ and $\omega^{n-1}=10^{10}$.
In this case, the eigenvalue problem is highly ill-conditioned
since the ratio of spectral spread and spectral gap is large: 
\begin{equation*}
\frac{\lambda_{n}-\lambda_{1}}{\lambda_{2}-\lambda_{1}}
= \frac{\omega^{n-1}-1}{\omega-1}\geq \omega^{n-2}\geq 10^{9}.
\end{equation*}
The vector $q$ for the auxiliary problem is constructed as
\begin{equation*}
 q = \eta[1,(\omega-1)^{2},\dotsc,(\omega-1)^{2n-2}]^{\Ttran},
\end{equation*}
where $\eta$ is a normalization parameter such that $\norm{q}=1$. In this case, the vector $q$ is super close to $u_{1}$ since $\rho_{q}-\lambda_{1}\approx 2\times 10^{-7}$.
In EPIC, the initial vector is chosen by $x_{0}=q$, the step--size $\tau_{P}$ is set as $\tau_{P}=\kappa_{P}^{-1/2}$, where $\kappa_{P}=L_{P}/\mu_{P}$ and the parameters $\mu_{P}$ and $L_{P}$ are selected by dropping the first order term of $\rho_{q}-\lambda_{1}$ in \Cref{defmuLNE}.
The stopping criteria are set as when the relative errors of approximate
eigenvalue are less than $10^{-14}$. 
The numerical results depicted in
\Cref{tabVarCon} are for
the parameters $\iota_{\nu} = (10k)^{2}$ with $k=1,2,\dotsc,12$.  The theoretical relationship between $\iota_{\nu}^{1/2}$ and iteration numbers $m_{P}$ in \Cref{NEkappa} is validated.

\begin{table}[htbp]
	\caption{Iteration numbers for different preconditioners.}
	\label{tabVarCon}
	\centering
	\begin{tabular}{|c|c|c|c|c|c|c|c|c|c|c|c|c|}
		\hline 
		$\iota_{\nu}^{1/2}$ &10 & 20 & 30 & 40 & 50 & 60 & 70 & 80 & 90 & 100 & 110 & 120 \\ \hline 
		$\#$ iter & 170 & 330 & 476 & 618 & 759 & 929 & 1074 & 1217 & 1351 & 1481 & 1612 & 1744\\ \hline
	\end{tabular}
\end{table}

\subsection{Test matrice from SuiteSparse Matrix Collection}
In this part, we compare the numerical behaviors between EPIC and LOPCG with test matrices $(A,M)$ listed in \Cref{testmat}. These matrices are from SuiteSparse Matrix Collection \cite{Davis2011}.

\begin{table}[htbp]
\centering
\caption{A list of test matrices.} \label{testmat}
\subfloat[standard eigenvalue problem $Au =u \lambda $]{
\begin{tabular}{cccc} \hline
Matrix          & Size   & nnz     & Application                          \\ \hline
\texttt{2cubes\_sphere}  & 101492 & 1647264 & Electromagnetics                     \\
\texttt{boneS01}  & 127224  & 5516602	  & Model Reduction Problem \\
\texttt{Dubcova3} & 146689  & 3636643  & 2D/3D Problem \\ 
\texttt{finan512}  & 74752  & 596992  & Economic \\ 
\texttt{G2\_circuit} & 150102  & 726674  & Circuit Simulation Problem \\ \hline
\end{tabular}} 

\medskip
\subfloat[generalized eigenvalue problem $A u =  Mu\lambda$]{
\begin{tabular}{cccc} \hline
Matrix              & Size & nnz        & Application             \\ \hline
(\texttt{bcsstk09},\texttt{bcsstm09}) & 1083 & (18437,1083)   & Structural Problem  \\
(\texttt{bcsstk21},\texttt{bcsstm21}) & 3600 & (26600,3600)   & Structural Problem  \\
(\texttt{Kuu},\texttt{Muu})           & 7102 & (340200,340200)& Structural Problem \\ \hline
\end{tabular}}
\end{table}

The vector $q$ is chosen as a random Gaussian vector with normalization. For both two methods, the initial vectors are set as $x_{0}=q$. 
Since the choice of $q$ will affect the behavior of EPIC, and the possibility of a random Gaussian vector satisfying the condition in \Cref{convexP} is extremely low, a restart strategy will be applied to EPIC. Specifically, when $\abs{x_{k}^{\Ttran}Mq}<0.5$, we will restart EPIC with $q=x_{k}$.

For the co--preconditioner $T$, we employ the aggregation-based algebraic multigrid preconditioner \cite{Notay2010}.
Different from the previous experiment, less attention will be paid to the choice of $\mu$ and $L$ in EPIC. We just set $\mu=L=6$ for all test matrices.

The stopping criteria of EPIC and LOPCG are chosen as when the relative errors of approximate eigenvalue are less than $10^{-8}$, \ie $\rho_{k}-\lambda_{1}\leq 10^{-8}\lambda_{1}$, where $\lambda_{1}$ is computed from Matlab's built-in function \texttt{eigs}. 

Numerical results are depicted in \Cref{ConHisSS,IterNumSS}.
We can see that the convergence histories of EPIC and LOPCG are very close, for both Rayleigh quotient and the components in $u_{1}$. 
In terms of the elapsed time per iteration, EPIC is slightly longer than LOPCG. 
We observe that the restart of EPIC only happens in the very early stage. 
For some hard example, such as \texttt{boneS01}, EPIC performs much better than LOPCG. 

\begin{figure}[htbp]
\centering
\subfloat[\texttt{2cubes\_sphere}]{
\includegraphics[width=\figsizeF]{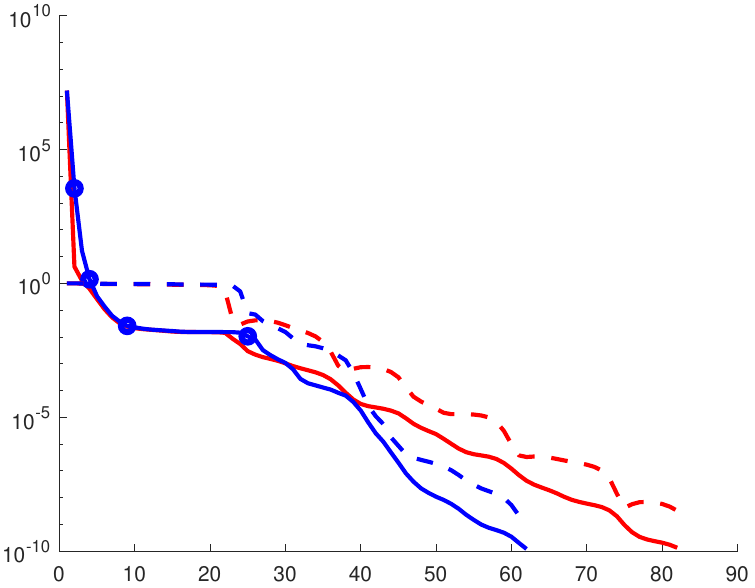}}
\subfloat[\texttt{boneS01}]{
\includegraphics[width=\figsizeF]{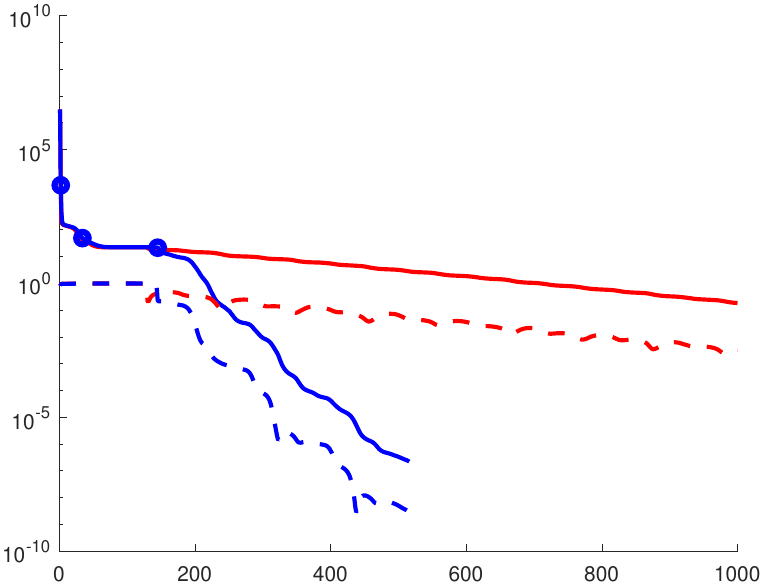}}
\subfloat[\texttt{Dubcova3}]{
\includegraphics[width=\figsizeF]{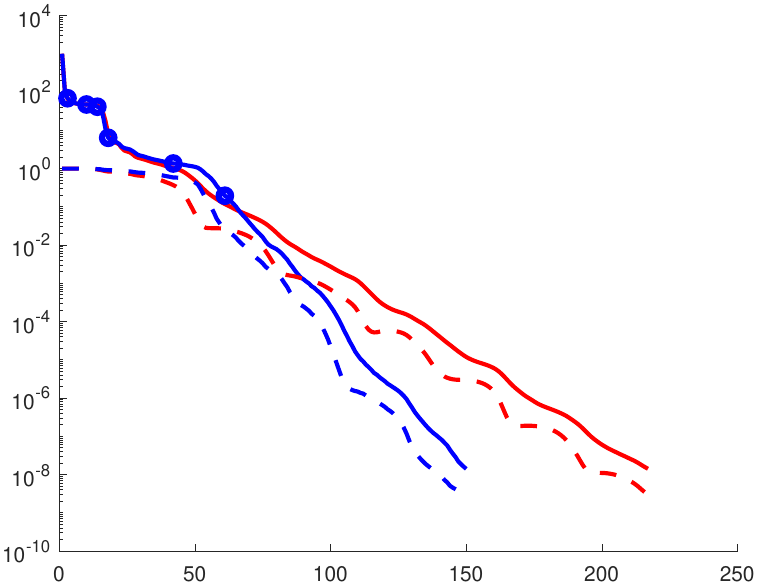}}
\subfloat[\texttt{finan512}]{
\includegraphics[width=\figsizeF]{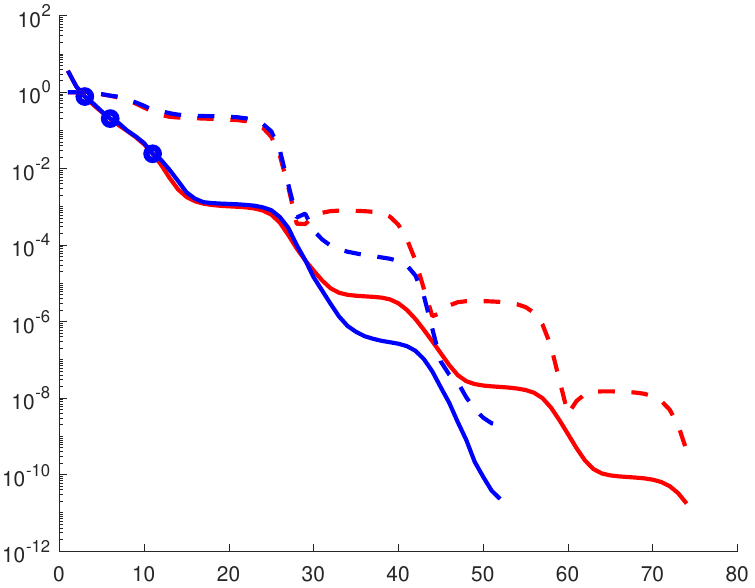}}\\ 

\subfloat[\texttt{G2\_circuit}]{
\includegraphics[width=\figsizeF]{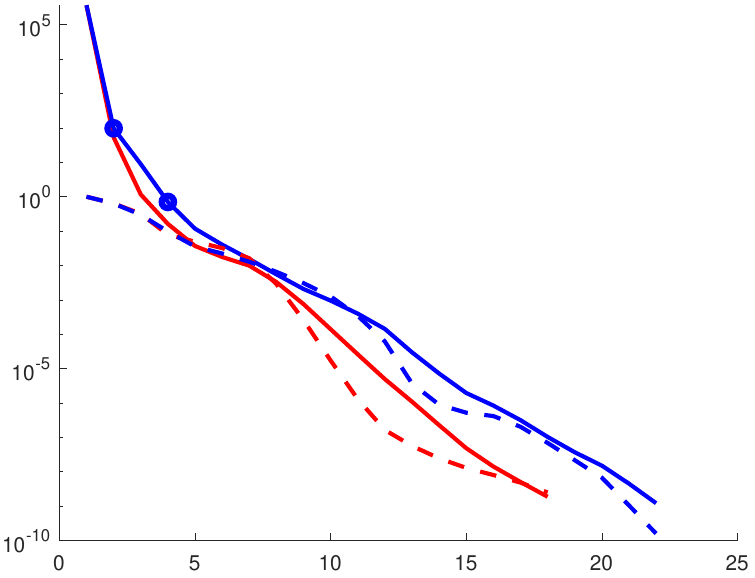}} 
\subfloat[\texttt{bcsst(k/m)09}]{
\includegraphics[width=\figsizeF]{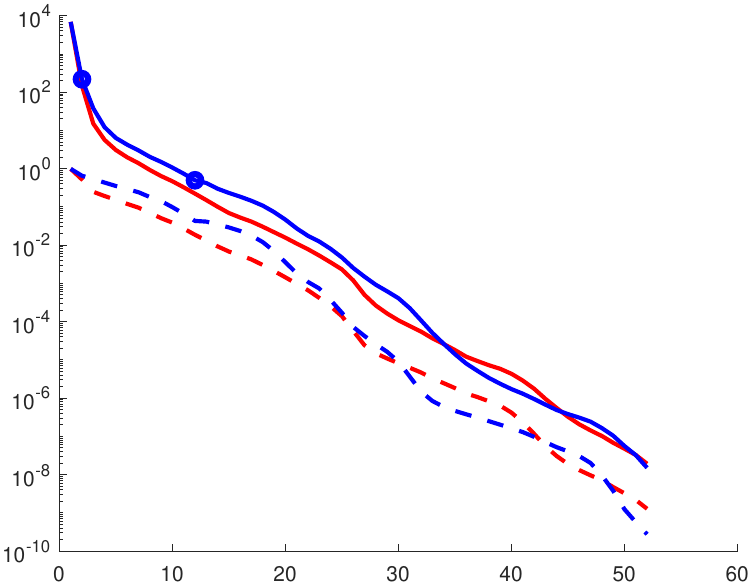}} 
\subfloat[\texttt{bcsst(k/m)21}]{
\includegraphics[width=\figsizeF]{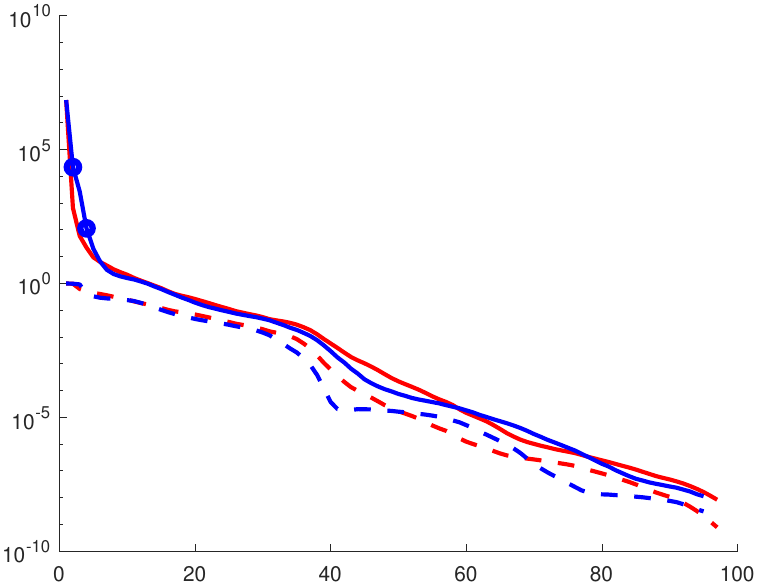}}
\subfloat[(K/M)uu]{
\includegraphics[width=\figsizeF]{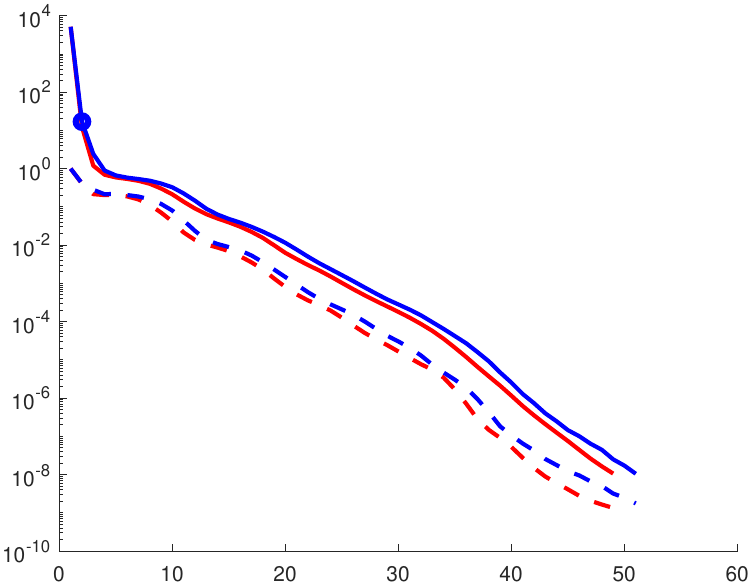}}
\caption{Convergence history of LOPCG (red) and EPIC (blue). The $x$-axis is the iterations number. The solid lines are the relative errors of approximate smallest eigenvalues, and the dashed lines are $1-\abs{x_{k}^{\Ttran}Mu_{1}}$, where $x_{k}$ and $u_{1}$ are both $M$--normalized. The restart points are marked by circle.} 
\label{ConHisSS}
\end{figure}

\begin{table}[htbp]
\caption{Iteration numbers and elapsed times. For the matrix \texttt{boneS01}, LOPCG does not converge in 1000 iterations.}
\label{IterNumSS}
\centering
\begin{tabular}{c|cc|cc} \hline
	Matrix              & LOPCG & time (\texttt{s})       & EPIC  & time (\texttt{s})     \\ \hline
	\texttt{2cubes\_sphere}      & 82 & 1.8985 & 62 & 1.6748 \\
	\texttt{boneS01}          & $\times$ & $\times$ & 516 & 27.0412 \\
	\texttt{Dubcova3}            & 217 & 7.7819 & 150 & 6.2790 \\
	\texttt{finan512}    & 74 & 0.9762 & 52 & 0.7907  \\  
	\texttt{G2\_circuit}  & 18 & 0.4572 & 22 & 0.6808     \\ \hline
	(\texttt{bcsstk09},\texttt{bcsstm09})& 52 & 0.1154 & 52 & 0.1209\\
	(\texttt{bcsstk21},\texttt{bcsstm21}) & 97 & 0.2354 & 95 & 0.2487\\
	(\texttt{Kuu},\texttt{Muu})          & 49 & 0.2667 & 51 & 0.3166 \\ \hline
	\end{tabular}
\end{table}


\section{Concluding remarks}
\label{sec:conclude} 
We introduced the concept of implicit convexity of the symmetric eigenvalue problem \eqref{evp}. A symmetric Eigensolver based on Preconditioning and Implicit Convexity (EPIC) with provable acceleration is proposed. Numerical results verify the theoretical rate of the convergence of the EPIC and show the similar rates of the convergence for EPIC and LOPCG for a set of test matrices from applications.

There are two research directions for future work. One is how to develop a parameter--free variant similar to LOPCG, and the other one is the development of a block version of the EPIC. 

\section*{Acknowledgments}
We thank the helpful discussion with Long Chen of UC Irvine. Part of this work was performed when the first author Shao was at School of Mathematical Sciences, Fudan University.

\bibliographystyle{siamplain}
\bibliography{ref}

\end{document}